\documentclass{amsart}
 
\usepackage{amsmath}
\usepackage{amssymb}
\usepackage{color}
\usepackage{a4wide}
\usepackage[all]{xy}

\newtheorem{theorem}{Theorem}[section]
\newtheorem{definition}[theorem]{Definition}

\newtheorem{lemma}[theorem]{Lemma}
\newtheorem{remark}[theorem]{Remark}
\newtheorem{proposition}[theorem]{Proposition}

\def\supp{\mathrm{supp}}

\def\to{\rightarrow}

\def\C{\mathbb C}
\def\R{\mathbb R}

\def\N{\mathbb N}
\def\al{\alpha}

\def\be{\beta}

\def\de{\delta}
\def\rh{\rho}
\def\et{\eta}

\def\ga{\gamma}
\def\GA{\Gamma}
\def\ve{\varepsilon}

\def\la{\lambda}

\def\om{\omega} 
\def\va{\varphi}

\def\ta{\tau}

\def\g{\mathfrak{g}}

\def\h{\mathfrak{h}}

\def\p{\mathfrak{p}}

\def\ga{\gamma}
\def\Ga{\Gamma}
\def\la{\lambda}
\def\ve{\varepsilon}
\def\si{\sigma}
\def\om{\omega}

\def\ph{\phi}
\def\ch{\chi}
\def\ta{\tau}
\def\ps{\psi}

\def\C{\mathbb{C}}

\def\Ad{\rm{\, Ad \,}}
\def\Om{\Omega}

\def\ol#1{\overline{#1}}
\def\nn{\nonumber}

\def\R{{\mathbb R}}
\def\C{{\mathbb C}}
\def\N{{\mathbb N}}

\def\T{{\mathbb T}}

\def\wh#1{\widehat{#1}}

  \def\Id{{\mathbb I}}
\def\A{{\mathcal A}}
\def\B{{\mathcal B}}

\def\F{{\mathcal F}}

\def\H{{\mathcal H}}
\def\K{{\mathcal K}}

\def\ZZ{{\mathcal Z}}
\def\OO{{\mathcal O}}

\def\iy{\infty}

\def\ol#1{\overline{#1}}

\def\hb#1{\hbox{#1}}
\def\val#1{\vert #1\vert}

\def\no#1{\Vert #1\Vert }
\def\noop#1{\Vert #1\Vert_{\rm op}}

\def\ind#1#2{\hb{ind}_{#1}^{#2}}

\def\supp#1{\text{supp}(#1)}

\def\limk{\lim_{k\to\infty}}
\def\res#1{_{\vert #1}}
\def\inv{^{-1}}

\def\hb #1{\hbox{#1}}

\def\hb#1{\hbox{#1}}
\def\val#1{\vert #1\vert}

\def\l#1#2{L^{#1}{(#2)}}

\def\lef({\left(}
\def\rig){\right)}

\def\Sig{\Sigma}

\def\supp{\mathrm{supp}}

\def\to{\rightarrow}

\begin{document}
\title{The C*-algebra of the Boidol group}
\author[Y.-F. Lin]{Ying-Fen Lin}
\address{Mathematical Sciences Research Centre,
  Queen's University Belfast, Belfast BT7 1NN, United Kingdom}
\email{y.lin@qub.ac.uk}

\author[J. Ludwig]{Jean Ludwig}
\address{Universit\'e de Lorraine, Institut Elie Cartan de Lorraine,  UMR 7502, Metz, F-57045, France}
\email{jean.ludwig@univ-lorraine.fr}

\keywords{C*-algebra, non-$\ast$-regular Lie group, algebra of operator fields, spectrum}

\begin{abstract}
The Boidol group is the smallest non-$\ast$-regular exponential Lie group. 
It is of dimension 4 and its Lie algebra is an extension of the Heisenberg Lie 
algebra by the reals with the roots $1$ and $-1$. We describe the C*-algebra of 
the Boidol group as an algebra of operator fields defined over the spectrum of 
the group. It is the only connected solvable Lie group of dimension less than or 
equal to 4 whose group C*-algebra had not yet been determined. 


\end{abstract}
\maketitle

\section{Introduction and notations}\label{intro}
Let $\A$ be a C*-algebra and $\widehat \A$ be its  spectrum. In order to 
analyse the C*-algebra $\A$, one can use the Fourier transform $\F$, which will 
allow us to decompose $\A$ over $\widehat \A$. To be able to define such a 
transform, we choose a representative $\pi$ in every equivalence 
class $[\pi]$ of $\widehat{\A}$ and we consider the C*-algebra $ l^\iy(\wh\A) $ 
of all bounded operator fields defined over $\widehat \A$ given by 
\begin{eqnarray}\label{}
 \nn  l^\iy(\widehat \A):= 
\{A=(A(\pi) \in \B(\H_\pi))_{[\pi]\in\widehat \A}; \no A_\iy:= 
\sup_{[\pi]}\noop {A(\pi)}<\iy\},
\end{eqnarray}
where $ \H_\pi $ is the Hilbert space of $\pi$. The Fourier transform $\F$ of $ \A $ is defined by 
\begin{eqnarray}\label{}
\nn \F(a):=\hat a:=(\pi(a))_{[\pi]\in\widehat \A} \quad \text{for} \quad a\in \A.
\end{eqnarray}
It is then an injective, hence isometric, homomorphism from $ \A $ into $l^\iy(\widehat\A)$. Therefore we can analyse the C*-algebra $ \A $ by recognising the elements of $\F(\A)$ inside the (big) C*-algebra $l^\iy(\widehat \A)$. However, for most C*-algebras its spectrum is not known or the topology of its spectrum is a mystery. 

In the case of the C*-algebra $C^*(G) $ of a locally compact group $G$, we know that the spectrum $ \widehat{C^*(G)} $ of $C^*(G) $ can be identified with the unitary dual $ \widehat G $ of $ G $. Furthermore, if $G$ is an \textit{exponential} Lie group, i.e., it is a connected, simply connected solvable Lie group for which the exponential mapping $\text{exp}: \g \to G$ from the Lie algebra $ \g $ to its Lie group $ G $ is a diffeomorphism, then the Kirillov-Bernat-Vergne-Pukanszky-Ludwig-Leptin theory shows that there is a canonical homeomorphism $ K: \g^*/G \to \widehat G $ from the space of coadjoint orbits of $ G$ in the linear dual space $ \g^* $ onto the unitary dual space $ \widehat G $ of $ G $ (see \cite{Lep-Lud} for details and references). In this case, one can identify the spectrum $ \widehat{C^*(G)} $ of the C*-algebra of an exponential Lie group $G$ with the space $ \g^*/G $ of coadjoint orbits of the group $ G $. Note that connected Lie groups are second countable, so the algebra $C^*(G) $ and its dual space $ \widehat G $ are separable topological spaces. 

The idea of describing group C*-algebras as algebras of operator fields defined on the dual spaces  was first introduced in 
\cite{Fe} and \cite{Lee1}. In serial, the C*-algebra of $ax+b$-like groups \cite{Lin-Lud}, of the Heisenberg groups and of the threadlike groups \cite{Lu-Tu}, of the {affine automorphism} groups $G_{n,\mu} $ in \cite{ILL1}, \cite{ILL2}, and of the group $\T\ltimes H_1 $ 
 \cite{Lu-Re} have all been characterised as algebras of operator fields defined 
on the corresponding spectrums of the groups. Note that each case has its own 
treatment due to the complexity of the coadjoint orbits of each group. In this 
way, the C*-algebra of every exponential Lie group of dimension less than or 
equal to $4 $ has been explicitly determined with one exception, namely 
\emph{the Boidol group}, which is an extension of the Heisenberg group by the 
reals with the roots $1 $ and $-1$.  
   
In this paper, we consider this  Boidol group $G= \exp \g$, which is the only
 non-$\ast$-regular exponential Lie group of dimension 4 (see 
~\cite{Boi2}).  We will write down precisely the dual space $\widehat G$ of Boidol's 
group $G= \exp \g$ using the structure of the coadjoint orbits in $\g^*/G$ and 
determine its topology. We decompose this orbit space into the union of a finite 
sequence $(\Ga_i=S_i\setminus S_{i-1})_{i= 0}^d$, where $d= 3$, of 
relatively closed subsets. On each of the sets $\Gamma_i$, the orbit space 
topology is Hausdorff and the main question is to understand the operator fields 
$ \hat a, a\in C^*(G)$, in 
particular the behaviour of the operators $\hat{a}(\gamma)$, $\gamma\in 
\Gamma_i$,  when $\gamma$ approaches elements in $\GA_{i-1}$. 
For each of these sets $\Gamma_i$, we obtain different conditions for the 
group C*-algebra. Since the spectrum of Boidol's group  has more layers 
than the spectra of the other groups of dimension less than or equal to $4 $, the  analysis  
of the behaviour of the operator fields $\hat a, a\in C^*(G), $ is more involved. 

We first recall the following definitions which were given in \cite{Be-Be-Lu}. Let $\H$ be a Hilbert space and $\B(\H)$ denote the algebra of bounded linear operators on $\H$.
 
\begin{definition}\label{lcd}\rm
Let $d$ be a natural number.
\begin{enumerate}\label{}
\item  Let $S$ be  a topological space. We say that $S $ is \textit{locally compact of step $\leq d$}, if there exists a finite increasing family $ \emptyset \ne  S_0\subset S_{1}\subset\cdots\subset S_d=S $ of closed subsets of $S $,  such that the subsets $\Gamma_0=S_0$ and $ \Gamma_i:=S_{i}\setminus S_{i-1}$,  $i=1,\dots, d$, are locally compact and Hausdorff in their relative topologies. 

\item Let $S$ be locally compact of step $\leq d$ and $\{\H_i\}_{i= 0, \dots, d}$ be Hilbert spaces. For a closed subset $ M \subset S $, denote by $ CB(M,\H_i) $ the unital C*-algebra of all uniformly bounded operator fields  $ (\psi(\gamma)\in \B(\H_i))_{\gamma\in M\cap \Gamma_i, i= 0,\dots, d}$, which are operator norm continuous on the subsets $ \Gamma_i \cap M$ for every $i\in\{0,\dots, d\} $ with  $ \Gamma_i\cap M\ne\emptyset $, such that $\ga\mapsto \psi(\ga) $ goes to 0 in operator norm if $\ga $ goes to infinity  
on $M$. We equip the algebra $ CB(M,\H_i) $ with the infinity-norm
\begin{align*}
\no{\psi}_{M}=\sup\left\{\no{\psi(\gamma)}_{\B(\H_i)}; M \cap \Gamma_i\ne \emptyset, \, \gamma\in M \cap \Gamma_i\right\}.
\end{align*}
\item  For every $s\in S $, choose a Hilbert space $\H_s $. We define the C*-algebra $l^\iy(S) $ of uniformly bounded operator fields defined over $S $ by
\begin{eqnarray*}
 l^\iy(S):=\{(\ph(s))_{s\in S}; \ \ph(s)\in  \B(\H_s), s\in S, \ \sup_{s\in S}\noop{\ph(s)}<\iy\}.
 \end{eqnarray*}
 \end{enumerate}
\end{definition}  


\begin{definition}\label{norcontspec}\rm
Let $\A $ be a separable liminary C*-algebra such that the spectrum $\widehat{\A} $ of $\A$ is a locally compact space of step $\leq 
d$, { with closed subsets $S_i$ of $\widehat{\A}$ and}
 \begin{align*}
  \emptyset= S_{-1}\subset  S_0\subset S_{1}\subset \cdots \subset S_d=\widehat{\A}. 
\end{align*}
Assume that for every $i \in \{0, \ldots, d\}$, there exists a Hilbert space $\H_i$ and a concrete realisation $(\pi_{\ga}, \H_i)$ of $\ga$ on the Hilbert space $\H_i$ for every $\ga\in \Gamma_i:= S_i\setminus S_{i-1}$.
Denote by $\F:\A\to l^\iy(\wh \A) $ the Fourier transform of $\A $ into the C*-algebra $l^\iy(\wh \A)$ defined as Definition \ref{lcd}(3), i.e., for $a\in \A $, let
\begin{eqnarray*}
 \F(a)(\ga)=\hat a(\ga):=\pi_\ga(a) \quad \text{ for } \ga\in \Ga_i \, \text { and } i= 0,\cdots, d. 
 \end{eqnarray*}
We say that $\F(\A )$ is \textit{continuous of step} $\leq d $, if the set $S_0 
$ is the collection of all characters of $\A$, and for every $ \gamma\in 
\Gamma_i $ there is a concrete realisation $(\pi_\gamma,\H_i)$ of $\gamma$ on the Hilbert space $ \H_i$ such that $\F(\A)\res{\GA_i}$ is contained in $CB({\GA_i},\H_i)$, for every $0\leq i\leq d $. 
 \end{definition}

\section{The Boidol group }\label{boidol group}

\subsection{}\label{Boidol def}
Let $\g $ be the real Lie algebra of dimension 4 with a basis $\{ T,X,Y,Z\} $ and the non-trivial brackets 
\begin{eqnarray*}
\nn  [ T,X]=  -X, \, [T,Y]= Y, \, [ X,Y]= Z.
 \end{eqnarray*}
The simply connected connected group $G $ with Lie algebra $\g $, which we call the Boidol group, can be realised on $\R^4 $ with the multiplication
\begin{eqnarray*}
 (t,x,y,z)\cdot (t', x' , y' ,{z'})=(t+t', e^{t'}x+x',e^{-t'}y+y', z+z'+\frac{1}{2}(e^{t'}xy'-e^{-t'}x'y)).
 \end{eqnarray*}  
The inverse of $(t, x, y, z) $ is given by 
\begin{eqnarray*}
 (t,x,y,z)\inv=(-t,-e^{-t}x, -e^{t}y,-z).
 \end{eqnarray*}
 We see that the subgroup $\ZZ:=\{ (0,0,0,z); z\in\R\} $ is the centre of $G $.
{Furthermore,
\begin{eqnarray*}
 G=\R\ltimes H
 \end{eqnarray*}
is the semi-direct product of $\R $ acting on the Heisenberg group $H= \{0\}\times \R^3$.}

\subsection{The coadjoint orbit space}
In this section we give a system of representatives of the coadjoint orbits in the linear dual space $\g^*$ of $\g$. Let $\{T^*, X^*, Y^*, Z^*\}$ be the dual basis of $\{ T,X,Y,Z\} $. We have three different kinds of coadjoint orbits in $\g^* $. 
\begin{enumerate}
\item The orbits in general position: For $(\rh, \la)\in \R\times \R^*$, where $\R^*= \R\setminus \{0\}$, let $\ell=\ell_{\rh,\la}=(\rh, 0, 0,\la)\in\g^* $. The stabiliser $G(\ell)  $ of $\ell $ in $G $ is given by $G(\ell)=\exp{\R T}\cdot \exp{\R Z}$, and the coadjoint orbit $O_{\rh,\la} $ of $\ell_{\rh,\la} $ is the subset  
 \begin{eqnarray*}\label{}
 \nn O_{\rh,\la}=  \big\{ (\frac{\rh\la +xy}{\la})T^*+xX^*+y Y^*+\la Z^*; x, y\in\R \big\}
 \end{eqnarray*}
of $\g^* $.

\item The orbits of dimension 2 vanishing on $Z $: For $(\al, \be)\in\R^2$ with $\al^2+\be^2\ne 0 $, the coadjoint orbit $O_{\al,\be,0}= {\Ad}^*(G)\ell_{\al,\be,0}$ of the element $\ell_{\al,\be,0}= \al X^*+\be Y^* $ is given by
 \begin{eqnarray*}\label{}
 \nn O_{\al,\be,0} &= & \{ u T^*+(e^t\al )X^*+(e^{-t} \be) Y^*; t,u\in\R\}\\
 \nn  &= & \{ u T^*+( \vert{\be }\vert {\al}e^t )X^*+ (sign(\be) e^{-t}) Y^*; t,u\in\R\}\\
\nn  &= & \{ u T^*+(sign({\al})e^t )X^*+( \vert{\al }\vert \be e^{-t}) Y^*; t,u\in\R\}. 
 \end{eqnarray*}

\item The real characters: For every $\ta\in\R $ we have the real  character $\ell_\ta:=\ta T^* $ of $\g $.
\end{enumerate}
\noindent This gives us the following partition of $\g^* $,
\begin{align*}
\g^*=\GA_3\bigcup\GA_2\bigcup\GA_1\bigcup\GA_0,
\end{align*}
where
\begin{enumerate}\label{}
\item[(i)] 
$ \GA_3=\{O_{\rh,\la}; \rh\in\R,\la\in\R^*\}.$
 The subset  
 \begin{eqnarray*}\label{}
 \nn \Sigma_3:=  \{ \ell_{\rh,\la};  (\rh,\la)\in \R\times \R^*\}
 \end{eqnarray*}
is a section for the $G $-orbits in $\GA_3 $.
\item[(ii)] Let 
\begin{align*}
 \GA_{2,\ve,\si}:= \{O_{\rh \ve,\si,0};  \rh\in\R_{+},\ve,\si\in\{+1,-1\}\},
 \end{align*}
and 
\begin{align*}\label{}
 \nn \GA_{2}:= \bigcup_{\ve,\si\in\{ +1,-1\}}\GA_{2,\ve,\si}.
 \end{align*}
For $\om\in\R^*$ and $\si \in\{+1,-1\} $, let
\begin{eqnarray*}
 \ell_{\om,\si,0}:=(0, \om,\si ,0)\in O_{\om,\si,0}.
 \end{eqnarray*}
 The subset $\Sigma_2$ of $\GA_2 $, defined by
 \begin{align*}\label{}
 \nn \Sigma_2:= \{ \ell_{\om,\si,0}; \si\in\{ +1,-1\},\om\in \R^*\},
 \end{align*}
is  a section for the $G$-orbits in $\GA_2 $.
 \item[(iii)]  
$ \GA_{1}:= \{O_{1,0,0},O_{-1,0,0}, O_{0,1,0},O_{0,-1,0} \}.$
The subset $\Sigma_1$ of $\GA_1$, given by 
 \begin{align*}\label{}
 \Sigma_1:= \{ \ell_{\si,0,0},\ell_{0,\si,0}; \si\in\{ +1,-1\}\},
 \end{align*}
  is a section for the $G$-orbits in $\GA_1 $.
   
 \item[(iv)] Finally let 
\begin{align*}
\Sigma_0= \GA_0:= \{\ell_\ta; \ta\in\R\},
 \end{align*}
where $\ell_\ta:=(\ta,0,0,0)$.

\noindent We recall that
\begin{eqnarray}\label{}
 \nn S_i= \bigcup_{j=0}^i \Ga_j, \quad i=0,1,2,3. 
\end{eqnarray}

 \end{enumerate}   

\begin{remark}\label{closed}
\rm
\begin{enumerate}
\item 
The spectrum of $G/\ZZ $ (and of $C^*(G/\ZZ) $) can be identified with the closed subset $\{ \pi\in\widehat G; \pi(\ZZ)=\{ \Id_{\H_\pi}\}\} $ and also with the subset $S_2= \{ \Om\in \g^*/G; \ell(Z)=0, \ell\in \Om\}= \GA_2\cup \GA_1\cup \GA_0 $ of the coadjoint orbit space.

\item Let $(\la_{G/\ZZ}, \l2{G/\ZZ} )$ be the left regular representation of the group $G/\ZZ$ and $(\tilde \la_{G/\ZZ}, \l2{G/\ZZ})$ be the corresponding representation of $G $. Since $G $ is amenable, the representation $\la_{G/\ZZ}$ is injective in $C^*(G/\ZZ) $. It is easy to see (either directly or using \cite{Reiter}) that $\tilde\la_{G/\ZZ}(C^*(G))= \la_{G/\ZZ}(C^*(G/\ZZ)) $. Furthermore the kernel of $\tilde \la_{G/\ZZ} $ in $C^*(G) $ is the ideal 
\begin{eqnarray*}\label{}
 \nn K_{S_2}= \{ a\in C^*(G); \pi(a)= 0 \, \text{ for } \pi\in S_2\}.
 \end{eqnarray*}
 The C*-algebra $C^*(G/\ZZ) $ is thus isomorphic to the quotient of $C^*(G)$ by $K_{S_2} $.

\item We observe that the group $G/\ZZ$ is an extension of $\R^2 $ by $\R$ with the roots $+1$ and $-1 $, and its C*-algebra has been described in \cite{Lin-Lud}.

\item It follows from the description of the coadjoint orbits that the orbits in $ \Ga_3$, $ \Ga_2$ and $ \Ga_0$ are closed in $ \g^*$, but 
the four orbits in $ \Ga_1$  are not. Since the canonical mapping $ K: \g^*/G\to \widehat{G} $  is a homeomorphism, it follows that for every closed coadjoint orbit $ \Om\in \g^*/G$  the irreducible representations $ \pi_\Om=K(\Om)$  associated to $ \Om$  send the C*-algebra of $ G$  onto the algebra of compact operators $ \K(\H_{\pi_\Om}) $ on the Hilbert space $\H_{\pi_\Om} $  of $ \pi_\Om$.
\end{enumerate}
\end{remark}

\begin{proposition}\label{toporbitspace}
The relative topology on $\GA_3 $ is Hausdorff. In $\GA_2 $ the subsets $\GA_{2,\ve,\si}$, $\ve, \si\in \{+1, -1\}$, are open and Hausdorff, $\GA_1 $ is discrete and $\GA_0 $ is homeomorphic to $\R $. 
\end{proposition}
 
\begin{proof}The proof follows easily from the fact that a sequence of coadjoint orbits $(O_k)_{k\in\N}\subset \g* $ converges to an orbit $O $ if and only if for every $\ell\in O $ and every $k\in\N $ there exists an $\ell_k\in O_k $ such that $\limk \ell_k=\ell $
(see \cite{Lep-Lud}, page 135). 
\end{proof}  

In the following we recall the definition of properly converging sequences and their limit sets, we also give a proposition of properly converging sequences in our group. 

\begin{definition}\label{def:properlyconverging}
\rm Let $S$ be a topological space. Let $\overline{x}= (x_k)_k$ be a net in $S$. We denote by $L(\overline{x})$ the set of all limit points of the net $\overline{x}$. A net $\overline{x}$ is called {\it properly converging} if $\overline{x}$ has limit points and if every cluster point of the net is a limit point, i.e., the set of limit points of any subnet is always the same, indeed, equals to $L(\overline{x})$.
\end{definition}

We know that every converging net in $S$ admits a properly converging subnet, hence, we can work with properly converging nets in our space.

\begin{proposition}\label{limtwothree}
Let $\ol O:= (O_{{\rh_k,\la_k}})_k$ be a properly converging sequence in $\GA_3$ such that $\limk \la_k= 0$. Then
\begin{enumerate}\label{}
\item the sequence $(\om_k)_k:=(\rh_k\la_k)_k $ converges to some $\om\in\R $;
\item if $\om\ne 0 $, then the limit set $L $ of the sequence $\ol O $ is the 
 two point set $\{O_{\om,- 1,0}, O_{-\om,1,0}\}\subset \GA_2 $;
\item if $\om=0 $, then the  limit set $L $ is the subset $\GA_1\bigcup \GA_0 $ of the orbit space. 
 \end{enumerate}
\end{proposition}

\begin{proof} 
(1) Suppose that the sequence $(O_{\rh_k,\la_k})_k\subset \GA_3  $ converges to some $O\in \GA_2\cup\GA_1\cup\GA_0 $. Take  $\ell= t T^*+xX^*+yY^*\in O $. Then for any $k\in\N $ we have an element $\ell_k= (\frac{\rh_k\la_k +x_ky_k}{\la_k})T^*+x_kX^*+y_k Y^*+\la_k Z^*\in O_{\rh_k, \la_k}$ which converges to $\ell $. In particular, $\limk y_k= y, \limk x_k= x $ and $\limk \frac{\rh_k\la_k +x_ky_k}{\la_k}= t $. Let $\om:=-xy $. Since $\lim_{k\to \infty} \la_k= 0 $, it follows that $\limk \rh_k\la_k= \om $.

(2) Suppose that $\om\ne 0$. Let $O $ be a limit point of the sequence $(O_{\rh_k,\la_k})_k $. Then $O$ is contained in $\GA_2\cup\GA_1\cup\GA_0$. Let $\ell= tT^*+xX^*+yY^*\in O $. Since $0 \ne \om= -xy $ we have that $O \in \GA_2 $. We can assume that $t= 0, \val x= \vert{\om }\vert  $ and $\val y= 1$. If $y= 1 $, then $x= -\om $ and conversely if $y= -1 $ then $x= \om$. 

(3) If $\om= 0 $ then $\limk x_ky_k= 0 $. Take $\ell= \si X^*\in \GA_1 $. We can use  $\ell_k= \si X^*-\si \om_k Y^*+\la_k 
Z^*= (\frac{\rh_k\la_k +\si (-\si \om_k)}{\la_k})T^*+\si X^*-\si \om_k Y^*+\la_k Z^*\in O_{\rh_k,\la_k}$ and $\om_k:=\rh_k\la_k$ for $k\in\N. $ Then $\limk \ell_k= \ell $. Similarly for $\ell= \si Y^* $. For $\ell= \al T^*\in \GA_0 $, we take $x_k:= - \sqrt{\val{\al \la_k-\om_k}}$ and $y_k:= \sqrt{\val{\al \la_k-\om_k}}sign(\al \la_k-\om_k)$, provided  $\al \la_k-\om_k\ne 0 $. Otherwise we use $x_k:= 0=: y_k$. Then $\ell_k= { \al T^*}+\la_k Z^*\in O_{\rh_k,\la_k}$ and $\limk \ell_k=\ell $.
\end{proof}

The following proposition can be found in \cite[Theorem~2.3]{Lin-Lud}.

\begin{proposition}\label{omiszero}
Let $\ol O= (O_{\ve\rh_k,\si,0})_k$, $\rh_k > 0$ and $\ve,\si\in\{ +1,-1\} $, be 
 a properly converging sequence in $\GA_2 $ with limits in $\GA_0\cup\GA_1 $. Then
 \begin{enumerate}\label{}
\item $\limk \rh_k= 0 $ and the limit set of $\ol O $ is $ \{O_{\ve,0,0},O_{0,\si,0}\}\cup \GA_0$. 
\item The closure of $O_{\si,0,0}$ (resp. $O_{0,\si,0}$) is $O_{\si,0,0}\cup 
\GA_0 $ (resp. $O_{0,\si,0}\cup \GA_0 $) for $\si= +1$ or $\si= -1 $. 
 \end{enumerate}
 \end{proposition}

Note that throughout the paper the symbols $ \si, \ve$ will denote elements of the set $\{+1,-1\} $.

\section{The Fourier transform }\label{Fourier}

 
 \begin{definition}\label{pi a defined}
\rm   
Let $(\pi,\H) $ be a representation of $G$ and ${{\al}}: G\to G $ be an automorphism. We define the new representation $(\pi^{{\al}} ,\H)$ of $G $ by 
\begin{eqnarray*}
 \pi^{{\al}}(g):= \pi({{\al}}\inv (g)) \, \text{ for } \, g\in G.
 \end{eqnarray*}
\end{definition}

\subsection{The irreducible representations} Let the automorphism $\al: G\to G$ be given by
 \begin{align*}\label{}
 \nn {{\al}}(t,x,y,z):=  (t,-x,-y,z).
 \end{align*}

$ {\bold{(I)}.}$ For $\ell_{\rh,\la}\in \GA_3 $, consider the Pukanszky polarisation $\p:= \text{span}\{T,Y,Z\}$ at $\ell_{\rh,\la} $. This gives us the irreducible representation 
\begin{eqnarray*}
 \pi_{\rh,\la}:= \ind P G\ch_{\rh,\la},
 \end{eqnarray*}
{where $P=\exp \p$ and $\ch_{\rh,\la}(\exp(t T)\exp(y Y)\exp(z Z)):= e^{-i (\rh t+ \la z )}$ for $t, y, z\in\R $} is the unitary character of $P $ corresponding to the character $\ell_{\rh, \la} $ of $\p$. The Hilbert space $ L^2(G/P, \ch_{\rh,\la}) $ of $\pi_{\rh,\la} $ can be realised as $\l2\R $. Let $(t, x, y, z)\in G, \ \xi\in {L^2(\R)}$ and $u\in\R $. We have
\begin{align*}  
 \pi_{\rh,\la}(t,x,y,z)\xi(u)= e^{t/2}e^{-i\rh t}e^{-i\la z}e^{-i{\la xy/2}{}}e^{i\la e^t yu}\xi(e^t u-x).
 \end{align*}
 For $F\in L^1(G) $, define its partial Fourier transform $\hat F^{3,4} $  by
 \begin{align*}\label{}
 \nn \hat F^{3,4}(t,x,a,b):= \int_{\R^2}F(t,x,y,z)e^{-i(ay+bz)}dydz \, \text{ for all } \, t, x, a, b\in \R.
 \end{align*}
Then for any $F\in\l1G$ such that $\hat F^{3,4} $ is contained in $C_c(\R^4)$, $ \xi\in L^2(\R)$ and $u \in \R$, we have that 
\begin{eqnarray}\label{kernel pirhola}
 \pi_{\rh,\la}(F)\xi(u)\nn  & = & \int_G F(g)\pi_{\rh,\la}(g)\xi(u)dg\\
 &=& {\int_\R \big(\int_\R e^{t/2} e^{-i\rh t}\hat F^{3,4}(t,x,\la(\frac{x}{2}-e^tu),\la)\xi(e^tu-x)dt\big)dx}\\  
 \nn & = & \int_\R \big(\int_\R e^{t/2} e^{-i\rh t}\hat F^{3,4}(t,e^tu-x,-\frac{\la}{2}(x+e^tu),\la)dt\big)\xi(x)dx.
 \end{eqnarray}
Therefore,  $\pi_{\rh,\la} $ is a kernel operator with kernel function 
\begin{eqnarray}\label{kerfunc rh la}
 F_{\rh,\la}(u,x)= \int_\R e^{t/2} e^{-i\rh t}\hat F^{3,4}(t,e^tu-x,-\frac{\la}{2}(x+e^tu),\la)dt \, \text{ for } \, u, x\in\R.
 \end{eqnarray}
 Furthermore, for $u\in \R$, we obtain the identity
\begin{align*}\label{}
 \nn \pi_{\rh,\la}^{{\al}}(t,x,y,z){ \xi}(u)= e^{t/2}e^{-i\rh t}e^{-i\la z}e^{-i\la xy/2}e^{-i\la e^t yu}\xi(e^t u+x),
 \end{align*}
which shows that 
\begin{eqnarray}\label{kernel pirhola s}
 \pi_{\rh,\la}^{{\al}} (F)\xi(u)\nn  & = & \int_{R^4} e^{t/2}e^{-i\rh t}e^{-i\la z}e^{-i{\la xy/2}{}}e^{-i\la e^t yu}F(t,x,y,z)\xi(e^t u+x)dzdydxdt\\
 & = & \int_\R \big(\int_\R e^{t/2} e^{-i\rh t}\hat F^{3,4}(t,-e^tu+x,\frac{\la}{2}(x+e^tu),\la)dt\big)\xi(x)dx.
 \end{eqnarray}
 Let $S: \l2\R \to \l2\R$ be the unitary operator defined by
 \begin{align*}\label{}
 \nn S(\xi)(u):= \xi(-u) \quad \text{for all} \quad \xi\in L^2(\R), u\in\R.
 \end{align*}
Then
\begin{eqnarray}\label{pilarh s}
 S\circ  \pi_{\rh,\la}(t,x,y,z)\circ S(\xi)(u)  &= & e^{t/2}e^{-i\rh t}e^{-i\la z}e^{-i\la xy/2}{{e^{-i\la e^t yu}\xi(e^t u+x)}}\\
  \nn&= & \pi_{\rh,\la}^{{\al}} (t,x,y,z)(\xi)(u), \quad (t,x,y,z)\in G.
 \end{eqnarray}

We shall need later in Section \ref{NCDL} an  equivalent version of the representations $\pi_{\rh,\la} $.

\begin{definition}\label{Vk pl defined}
\rm   
 For any $(\rho, \la)\in \R\times \R^*$ and measurable function $\va: \R \to \R$, the operator $V: 
\l2{\R, \frac{dx}{\vert x\vert }} \to \l2{\R,dx}$ defined by
\begin{eqnarray*}
 V(\et)(s)= \xi(s):= \frac{1}{\vert s\vert ^{1/2}}\et(\vert \la \vert s) e^{ i\va({ \la s )}}, \quad \et\in\l2{\R, \frac{dx}{\val x}}, \, s\in \R,
\end{eqnarray*}
is a unitary operator. We have that
\begin{align*} 
 V^*(\xi)(u)= \frac{\vert u\vert ^{1/2}}{\sqrt{\val {\la}}}e^{- i\va(u)}\xi\big(\frac{u}{\vert \la\vert}\big), \quad u\in\R, \, \xi\in \l2 {\R}. 
\end{align*}
We can also view $V$ as an operator $V: \l2{\R_+, \frac{dx}{\val x}}\oplus \l2{\R_-, \frac{dx}{\val x}}\to \l2{\R_+, {dx}}\oplus  \l2{\R_-, {dx}}$. It is easy to see that 
\begin{align}\label{ScommVk}
  S\circ  V= V\circ  S, 
\end{align}
where $S: \l2{\R_\si, \frac{du}{ \vert{u }\vert }}\to \l2{\R_{-\si}, \frac{du}{ \vert{u }\vert }} $ is defined as before by
\begin{align*}\label{}
 \nn  S(\et)(u)=  \et(-u), \quad u\in\R_\si.
 \end{align*}
\end{definition}

Let us now take the measurable function $\va$ as 
\begin{eqnarray}\label{vak def}
 \nn \va(s):= e^{i\rh\ln( \vert{s }\vert )}, \, s\in\R. 
\end{eqnarray}
For $ \et \in C_c^\iy(\R_+), s\in\R$ and $ \hat F^{3,4}\in C_c^\iy (\R^4) \subset \l1G$,  we see by {(\ref{kernel pirhola})} that
{\begin{eqnarray}\label{rh al circ Vk}
\pi_{\rh, \la}(F)(V(\et))(s)
\nn  &= & \int_\R \int_\R e^{t/2} e^{-i\rh t}\hat F^{3,4}(t,x,-\frac{\la}{2}(-x+2e^ts),\la)\\
 \nn & & \frac{e^{i\rh \ln(\val{-\la x+\la e^ts})}\val{\la}^{1/2}}{\val{-\la x+e^t \la s }^{1/2}}\et(-\vert \la\vert  x+\vert \la\vert  e^t s)dxdt.
 \end{eqnarray}}
Hence, writing $\la=\ve \vert{\la}\vert $, we obtain:
 {\begin{eqnarray}\label{vkstar circ Vk}
 \nn V^*(\pi_{\rh, \la}(F)(V(\et)))(s)
 &= &
 {e^{-i\rh \ln(\val s)}}\int_\R \int_\R  e^{-i\rh t}\hat F^{3,4}(t, x, \frac{\la}{2}x- \ve {e^ts}{},\la)\\
 \nn & & \frac{e^{i\rh \ln(\val{-\la x+  \ve e^ts})}\val{e^{t}s}^{1/2}}{\val{-\ve\la x+  e^ts }^{1/2}}\et(- \ve\la x+ e^t s )dxdt\\  
&= & \nn
\int_\R \int_\R  \hat F^{3,4}(t, x, \frac{\la}{2}x- \ve{e^ts}{},\la)\\
\nn & & 
\frac{e^{i\rh(\ln(\val{-\ve\la x+ e^ts}- \ln(\val{e^{t}s}))}\val{e^{t}s}^{1/2}}{\val{-\ve\la x+ e^ts}^{1/2}}\et(-\ve\la x+ e^t s)dxdt\\
&= & 
\nn{e^{-i\rh \ln(\val s)}}\int_\R \int_\R  e^{-i\rh t}\hat F^{3,4}(t, \frac{1}{\la}(\la x+\ve{e^ts}),\frac{1}{2}(\la x-\ve {e^ts}), \la)\\
 \nn & & \frac{e^{i\rh \ln(\val{\la x})}\val{e^{t}s}^{1/2}}{\val{\la x}^{1/2}}\et(- \ve\la x)dxdt\\
\nn &= &
 \frac{e^{-i\rh \ln(\val s)}\val{s}^{1/2}}{ \vert{\la}\vert}{}\int_\R \int_\R  e^{-i\rh t}{e^{t/2}}\hat F^{3,4}(t,\frac{-\ve}{\la}( x-{e^ts}),\frac{-\ve}{2}(x+{e^ts}),\la)dt\\
 \nn & & \frac{e^{i\rh \ln(\val{x})}}{\val{x}^{1/2}}\et( x)dx\\
&= & 
\frac{e^{-i\rh \ln(\val s)}\val{s}^{1/2}}{ \vert{\la}\vert}{}\int_\R H(s,x)\frac{e^{i\rh \ln(\val{x})}}{\val{x}^{1/2}}\et(x)dx,
 \end{eqnarray} 
 }
 where
 \begin{eqnarray*}\label{}
 \nn H(s, x) &:= &
 \int_\R  e^{-i\rh t}{e^{t/2}}\hat F^{3,4}(t,\frac{-\ve}{\la}( x-{e^ts}), \frac{-\ve}{2}(x+{e^ts}), \la)dt, \quad x, s\in \R.
 \end{eqnarray*}
Using partial integration, we see that
\begin{eqnarray*}\label{}
 \nn H(s, x) &= &
\frac{1}{i\rh} \int_\R e^{-i\rh t}\partial_t\Big({e^{t/2}}\hat F^{3,4}(t,\frac{-\ve}{\la}( x-{e^ts}),\frac{-\ve}{2}(x+{e^ts}),\la)\Big)dt\\
\nn  
&= &
\frac{1}{i\rh}\int_\R e^{-i\rh t} \Big(\frac{1}{2}e^{t/2} \hat F^{3,4}(t,\frac{-\ve}{\la}( x-{e^ts}),\frac{-\ve}{2}(x+{e^ts}),\la)\\
 \nn  & &+ e^{t/2} \ve\frac{e^ts}{\la} \partial_2\hat F^{3,4}(t,\frac{-\ve}{\la}( x-{e^ts}),\frac{-\ve}{2}(x+{e^ts}),\la)\\
\nn  & &- e^{t/2}\frac{\ve}{2}  {e^ts}\partial_3\hat F^{3,4}(t,\frac{-\ve}{\la}( x-{e^ts}),\frac{-\ve}{2}(x+{e^ts}),\la)\Big)dt.
\end{eqnarray*}
Hence, there exist positive continuous functions $\psi, \psi'$ with compact support on $\R$ such that
 \begin{eqnarray}\label{Hintro}  
\nn \vert{H(s, x) }\vert &\leq &
 \frac{1}{ \vert{\rh \la}\vert }\psi'(\la)(1+ \vert{s }\vert )\int_\R \psi'(t) \psi'(\frac{x-e^ts}{\la}) \psi'(x+e^ts)dt\\
&\leq&
\frac{1}{ \vert{\rh \la}\vert }\psi(\la)\int_\R \psi(t)\psi(\frac{x-e^ts}{\la})\psi(x+e^ts)dt, \quad s, x\in\R.
 \end{eqnarray}
Furthermore, with the automorphism $ \al$ on the group $G$ and $s\in\R $, we have that
 \begin{eqnarray}\label{vkstar s circ Vk}
 V^*(\pi_{\rh, \la}^{{\al}} (F)(V(\et)))(s)
 \nn  &= & {e^{-i\rh \ln(\val s)}}\int_\R \int_\R  e^{-i\rh t}\hat F^{3,4}(t, x, \frac{\la}{2}x+\ve {e^ts}{},\la)\\
 \nn & & \frac{e^{i\rh \ln(\val{\la x+  \ve e^ts})}\val{e^{t}s}^{1/2}}{\val{\ve\la x+  e^ts }^{1/2}}\et(\ve\la x+  e^t s )dxdt\\
\nn  &= & \int_\R \int_\R  \hat F^{3,4}(t, x, \frac{\la}{2}x+\ve{e^ts}{},\la)\\
 \nn & & \frac{e^{i\rh (\ln(\val{\la x+  \ve e^ts}- \ln(\val{e^{t}s}))}\val{e^{t}s}^{1/2}}{\val{\la x+ \ve e^ts}^{1/2}}\et(\ve\la x+ e^t s)dxdt\\
 &= & \int_\R \int_\R  \hat F^{3,4}(t,-x,-\frac{\la}{2}x+\ve{e^ts}{},\la)\\
 \nn & & \frac{e^{i\rh (\ln(\val{-\la x+  \ve e^ts}- \ln(\val{e^{t}s}))}\val{e^{t}s}^{1/2}}{\val{-\la x+  \ve e^ts}^{1/2}}\et(-\ve\la x+ e^t s)dxdt.
 \end{eqnarray}

 ${\bold{(II)}.}$ For $\ell\in\GA_2 \bigcup \GA_1$, the subalgebra $\h:= \text{span}\{X,Y,Z\}$ is a Pukanszky polarisation at $\ell$. Therefore the unitary representation 
\begin{eqnarray*}
 \pi_{\ell}:=\ind H G\ch_{\ell},
 \end{eqnarray*}
where $H:= \exp\h $ and $\ch_{\ell}(\exp(U)):= e^{-i \ell(U)}$ for $U\in \g$, is irreducible.

Take $\ell:= (0,\mu,\nu,0) $ for some $\mu,\nu\in\R $ with $\mu^2+\nu^2\ne 0$. We have that $L^2(G/H,\chi_\ell)\simeq\l2\R $, and for $(t,x,y,z)\in G, \va\in L^2(\R) $ and $v\in\R $:
\begin{align*}
 \pi_{\ell}(t,x,y,z)\va(v) = e^{-i\mu e^{v-t}x}e^{-i \nu e^{t-v}y}\va(v-t).
 \end{align*}
Hence for $F\in\l1G $,
\begin{align*}
 \pi_{\ell}(F)\va(v) = \int_\R \hat F^{2,3,4}(v-t,\mu {e^{t}},\nu 
{e^{-t}},0)\va(t)dt.
 \end{align*}
This means that  $\pi_\ell(F) $ is a kernel operator with kernel function
\begin{eqnarray}\label{kerfunc mu nu}
 F_{\ell}(u,t)= \hat F^{2,3,4}(v-t,\mu {e^t},\nu {e^{-t}},0) \quad \text{for } \, v, t\in \R.
 \end{eqnarray}
Let us write an equivalent representation for $\pi_{\ell}$: We use the multiplication invariant measure $\frac{du}{\val u} $  on $\R_{\si}$, $\si\in\{ +1,-1\} $. Let 
\begin{eqnarray*}
 U_\si: \l2{\R} \to \l2{\R_\si,\frac{du}{\vert u\vert }} \, \text{ by } \, 
U_\si\xi(u):= \xi(-\ln (\si u)), \, u\in\R_\si.
 \end{eqnarray*}
 Then $U_\si  $ is a unitary operator and 
\begin{eqnarray*}
 (U_\si)^*(\et)(s)= \et(\si e^{- s}) \, \text{ for } \, s\in\R, \et\in 
\l2{\R_\si, \frac{du}{\val u}}.
 \end{eqnarray*}
Let 
\begin{eqnarray}\label{ta mu nu def} 
 \ta^{\si}_{\mu,\nu}:=U_\si\circ \pi_{(0,\mu,\nu,0)}\circ U^*_\si. 
\end{eqnarray}
We obtain the relation:

\begin{eqnarray*}
 \ta_{\mu,\nu}^{\si}(t,x,y,z)\et(u)\nn 
 &= & (\pi_\ell(t,x,y,z)(U_\si^*(\et)))(-\ln(\si u))\\
 \nn  &= & e^{-i\mu e^{-\ln(\si u)-t}x}e^{-i \nu e^{t+\ln(\si 
u)}y}U_\si^*(\et)(-\ln(\si u)-t)\\  
\nn  &= & e^{-i(\si\mu) u\inv e^{-t}x}e^{-i (\si\nu )u e^{t}y}\et(e^tu), \quad 
\et\in\l2{\R_\si, \frac{du}{\val u}}.
 \end{eqnarray*}
 We see that
 \begin{eqnarray}\label{equi - + s} 
 (\ta^{\si}_{\mu,\nu}){{^\al}} (t,x,y,z)\et(u)&=& e^{i(\si\mu) u\inv 
e^{-t}x}e^{i (\si\nu )u e^{t}y}\et(e^tu)\\
 \nn  &= & (\ta^{\si}_{-\mu,-\nu})(t,x,y,z)\et(u), \quad \et\in\l2{\R_\si, 
\frac{du}{\val u}}.
\end{eqnarray}
On the other hand,
\begin{eqnarray}\label{tamunu circ S}
  S\circ \ta^\si_{\mu,\nu}(t,x,y,z)\circ  S(\et)(u) &= & e^{-i(-\si\mu) u\inv 
e^{-t}x}e^{-i (-\si\nu )u e^{t}y}\et(e^tu)\\
 \nn&=& \ta^{-\si}_{\mu,\nu}(t, x, y, z)\et(u), \quad \et\in\l2{\R_{-\si}, 
\frac{du}{\val u}}, u\in \R_{-\si}.
\end{eqnarray}
In particular the representations $ \ta^\si_{\mu,\nu}$ and $\ta^{-\si}_{\mu,\nu}$ are equivalent. Furthermore, for $\et\in L^2(\R_\si, \frac{du}{|u|})$ and $F\in L^1(G) $:
 \begin{eqnarray}\label{ta mu nu}
 \ta_{\mu,\nu}^{\si}(F)\et(u) &= & \int_G F(t,x,y,z)e^{-i\si\mu u\inv e^{-t}x}e^{-i \si\nu u e^{t}y}\et(e^tu)dtdxdydz\\ 
\nn  &= & \int_\R \hat F^{2,3,4}(t,\si\mu (e^{t}u)\inv ,\si\nu (e^tu), 0)\et(e^tu)dt\\
\nn  &= & \int_\R \hat F^{2,3,4}(t-\ln(\si u),{ \mu e^{-t} ,\nu e^t}, 0)\et(\si e^t)dt. 
 \end{eqnarray}
 
 ${\bold {(III)}.}$ For $\ell\in \GA_0 $, let
\begin{eqnarray*}
 \pi_{\ell}:=\ch_{\ell}.  
 \end{eqnarray*}
Hence for $F\in\l1G $,
\begin{eqnarray*}
 \pi_{\ell}(F)= \int_\R\int_{\R^3}e^{-it\ell}F(t, h)dhdt.
 \end{eqnarray*}

In order to define the Fourier transform $a\mapsto \hat a, a\in C^*(G)$, we 
identify $\widehat G $ with the set $\GA:= \GA_3\cup \GA_2 \cup \GA_1\cup\GA_0$ 
and we let :
 \begin{eqnarray}\label{fourrier def}
  \hat a(\rh,\la) &:= & \pi_{\rh,\la}(a)\in\K(L^2(\R)),\, \, \rh\in \R,\la\in\R^*,\\
 \nn  \hat a(\mu,\nu,0)&:= & \ta^+_{\mu,\nu}(a)\in \K(L^2(\R_+, \frac{du}{u})), \, \, \ell_{\mu,\nu,0}\in{{\Gamma_2}},\\
 \nn  \hat a(\mu,\nu,0)&:= & \ta^+_{\mu,\nu}(a)\in \B(L^2(\R_+, \frac{du}{u})), \, \, \ell_{\mu, \nu,0}\in {{\Ga_1}},\\ 
\nn  \hat a(\ell)&:=&\ch_\ell(a)\in \C,\, \, \ell\in \GA_0.
 \end{eqnarray}
 

\section{Norm control of dual limits. }\label{NCDL}

In this section, we will describe the conditions that our C*-algebra as the image of the Fourier transform must fulfil, and characterise the C*-algebra of the Biodol group which will be our main result (Theorem~\ref{Cster G identified}).

\subsection{Norm convergence }\label{norm conv}$ $
\begin{definition}\label{toinfty}
\rm   Let $\OO=(O_k)_k $ be a sequence in $\g^* $. We say that $\OO $ \textit{tends to infinity}, if for every bounded subset $B\subset \g^* $, the subsets $O_k\cap {\Ad}^*(G)B$ of $\g^* $ are empty for $k $ large enough. This is equivalent to the property that no sequence $(l_k\in O_k)_{k\in \N} $ admits a convergent subsequence.
 \end{definition}

We have the following norm convergence condition. 

\begin{proposition}\label{normcontinuity}
For any $a\in C^*(G)$, the mappings $O\mapsto \pi_O (a)$ are norm continuous on the different sets $\GA_i$ for $i= 0, 1, 2, 3$. For any sequence $(O_k)_k $ tending to infinity, we have that $\limk \noop {\pi_{O_k}(a) }= 0$.  
\end{proposition} 

\begin{proof} It suffices to consider only $L^1 $-functions $F $ for which the partial Fourier transforms $\hat F^{3,4} $ are contained in $C_c^\iy(\R^4) $, since this space is dense in $C^*(G) $. The expression $(2) $ tells us that the operators 
$\pi_{\ell}(F)$, $\ell \in \GA_3$, are Hilbert-Schmidt. Indeed for those functions $F $, the kernel functions $F_{\rh,\la}(u,x) $ are continuous in the parameters $\rh,\la,u,x $ and have compact support in  $u,x\in\R $ and $\val \la> c$ (for some fixed positive real number $c$). Hence for any converging sequence $(\ell_k)_k\subset \GA_3$ with limit $\ell \in \GA_3$, the operators $\pi_{\ell_k}(F) $ converge in the Hilbert-Schmidt norm to the operator $\pi_{\ell}(F) $. The cases in $\GA_2, \GA_1$ and $\GA_0$ are treated in \cite{Lin-Lud}.

If now we have that $\limk O_k=\iy $ in $\GA_3 $,  then according to Proposition~\ref{limtwothree} we have that $\lim_{k\to\iy} \rh_k 
\la_k=\iy$. If $\limk \la_k=\iy $, then  $\pi_{\rh_k,\la_k}(F)=0 $ for $k $ large enough, and if $\limk \rh_k=\iy  $, then $\vert F_{\rh_k,\la_k}(x,u)\vert \leq \frac{1}{ \vert{\rh_k }\vert }\ps(x-u)$ for some positive $\ps\in C_c(\R)$ and any $x,u\in\R $, hence $\limk \noop{\pi_{\rh_k,\la_k}(F)}=0$. In $\GA_0, \GA_1$ and $\GA_2$, the property of sequences of coadjoint orbits tending to infinity has been described in \cite{Lin-Lud}.
\end{proof}

\subsection{An approximation of $\pi_{\rh_k,\la_k}(F) $}\label{approx}$ $

Let $(\pi_{\rh_k,\la_k} )_k$ be  a properly converging sequence in $\wh G$ such that $\limk \la_k= 0 $ and that $\limk \la_k\rh_k= \om \in \R $. We can assume (by passing to a subsequence if necessary) that $\la_k= \ve \val{\la_k} $ for every $k\in\N$. 

 \begin{definition}\label{R_k Mde def}
\rm Let $c\leq d\in\lbrack -\iy,\infty \rbrack $ and $\de> 0 $. We define a family of multiplication operators on the space $\l2\R$ or $\l2{\R_{\pm},\frac{dx}{ \vert x \vert }}$: for $s\in\R$ and $\xi\in L^2(\R)$ (or $\l2{\R_{\pm},\frac{dx}{ \vert x \vert }}$),
\begin{eqnarray*}
M_{\{c\leq d\}}(\xi)(s)\nn  
&:= &
1_{ [ c, d]}(s)\xi(s),\\ 
M_{\{\leq d\}}(\xi)(s)\nn 
&:= &
1_{ [ -\iy, d]}(s)\xi(s),\\  
M_{\{ d\leq \}  }(\xi)(s)\nn  
&:= &
1_{\lbrack d,\iy \rbrack}\xi(s),\\  
 M_{\{\val{}\leq \de\}}(\xi)(s)
 &:=&
 1_{{[ -\de,\de]}}(s)\xi(s), \\
 \nn \text{resp. }M_{{\{\de\leq \val{}}\}}(\xi)(s) 
 & :=&
1_{\lbrack -\iy,-\de \rbrack\cup \lbrack \de,\iy \rbrack}(s)\xi(s).
\end{eqnarray*}
 \end{definition}

\begin{lemma}\label{limk de=0rk to iy}
Suppose that $\limk \la_k= 0 $. Take a sequence $(R_k)_k $ in $\R_+ $ such that 
 \begin{eqnarray}
\nn \limk R_k= +\iy.
 \end{eqnarray}
Then we have that 
\begin{eqnarray}\label{small to 0}
 \limk \noop{V_k^*\circ \pi_{\rh_k,\la_k}(a)\circ V_{k }\circ M_{\{   R_k \leq 
\vert{ }\vert\} }}= 0, \quad a\in C^*(G).
 \end{eqnarray}
 \end{lemma}

 \begin{proof}
By (\ref{kernel pirhola}), it suffices to show that for $F\in \l1{G} $ with $\hat F^{3,4}\in C_c(\R^4) $ we have 
\begin{eqnarray}\label{hat F 34 is 0}
 \hat F^{3,4}\big(t,\ve\frac{e^ts}{\la_k}-\frac{\ve x}{\la_k},-\frac{1}{2}(\ve x+\ve{e^ts}{}), \la_k\big)= 0,
 \end{eqnarray}
for $\vert x\vert\geq R_k $ and $k $ large enough. Now if $s, x$ have the same sign, we have that $\vert \ve x+\ve{e^ts} \vert \geq R_k $ and thus (\ref{hat F 34 is 0}) is satisfied. Similarly, if $s $ and $x $ have different signs, it follows that $\vert \frac{e^ts}{\la_k}- \frac{x}{\la_k}\vert\geq \frac{R_k}{\val{\la_k}} $.
\end{proof}

\begin{lemma}\label{limk de=0}
Suppose that $\limk \la_k= 0 $ and $\limk{ {\rh_k\la_k}}=\om\in\R^*$. Take a sequence $(R_k)_k $ in $\R_+ $ such that  
\begin{eqnarray}\label{Rk conditions 1- nonzero om}
\limk R_k= +\iy \ \ \text{and} \ \  \limk R_k\vert \la_k\vert = 0.
 \end{eqnarray}
 Then we have that 
  \begin{eqnarray}\label{small to 0}
 \limk \noop{\pi_{\rh_k,\la_k}(a)\circ V_{k }\circ M_{\{ \vert{ }\vert\leq \vert R_k\la_k\vert\} }}= 0, \quad a\in C^*(G).
 \end{eqnarray}
 \end{lemma}

\begin{proof} 
We prove the lemma first for $F\in\l1G $ such that $\hat F^{3,4}\in C_c^{\iy}(\R^4) $. Let $M_{I_k}:= M_{\{\val{}\leq R_k\val{\la_k}\}}$ be the multiplication operator on $L^2{(\R,\frac{dx} {|x|}) }$ with the characteristic function of the interval $I_k:=[-R_k \vert{\la_k }\vert, R_k \vert{\la_k }\vert ] $. Applying the relation (\ref{vkstar circ Vk}) with sequences $(\rho_k)_k$ and $(\la_k)_k$ in $\R$, the operator $\pi_{\rh_k,\la_k}(F)\circ V_{k }\circ M_{I_k} $ is a kernel operator with kernel $H_k(s,x)1_{I_k}(x)$, and 
\begin{eqnarray*}\label{}
 \nn  \vert{H_k(s,x)1_{I_k}(x)}\vert \leq 1_{I_k}(x)\frac{1}{ \vert{\rh_k\la_k }\vert }\psi(\la_k)\int_\R\psi(t)\psi(\frac{x-e^ts}{\la_k})\psi(x+e^ts)dt,
 \end{eqnarray*}
for $s, x\in \R$ and $k\in\N$, where $\psi$ is a positive continuous function with compact support on $\R$ given in \eqref{Hintro}. 
Now for any $C>0 $ large enough, we have that $ \vert{\frac{x-e^ts}{\la_k}}\vert \geq R_k $ if $t\in\supp(\psi)$, $ \vert{x }\vert \leq R_k 
\vert{\la_k}\vert$, and $\vert{s }\vert \geq CR_k \vert{\la_k }\vert$. Therefore, it follows that  
\begin{eqnarray*}\label{}
 \nn \sup_{x\in I_k}\int_{\R} \vert H_k(s,x)\vert ds &\leq &
 \frac{1}{ \vert{\rh_k\la_k }\vert }\no{\psi}_\iy^2\sup_{x\in I_k}\int_{\R} \int_\R \psi(t)\psi(\frac{x-e^ts}{\la_k})dtds\\
\nn  
&= &
\frac{1}{ \vert{\rh_k\la_k }\vert }\no{\psi}_\iy^2\sup_{ x\in I_k}\int_{ C I_k} \int_\R \psi(t)\psi(\frac{x-e^ts}{\la_k})dtds\\
\nn  
&\leq &
C' R_k  \vert{\la_k }\vert 
 \end{eqnarray*}
for a new constant $C'>0 $ (independent of $k $).
Furthermore,
 \begin{eqnarray*}\label{}
 \nn \sup_{s\in \R}\int_{\R} \vert H_k(s,x)\vert 1_{I_k}(x)dx &\leq &
 \frac{1}{ \vert{\rh_k\la_k }\vert }\no{\psi}_\iy^2\sup_{s\in \R}\int_{I_k} \int_\R \psi(t)\psi(\frac{x-e^ts}{\la_k})dtdx\\
\nn  
&\leq&
C'' R_k  \vert{\la_k }\vert 
 \end{eqnarray*}
 for another constant $C''> 0$. Hence, Young's condition tells us that 
 \begin{eqnarray*}\label{}
 \nn \limk \noop{\pi_{\rh_k,\la_k}(F)\circ V_{k }\circ M_{\{ \vert{ }\vert\leq \vert R_k\la_k\vert\} }}= 0.
 \end{eqnarray*}
Now if we take any $a\in C^*(G) $, for every $\ve>0 $ there exists an $F_\ve $ in $L^1(G) $ with the properties above such that $\no{a-F_\ve}_{C^*(G)}<\ve $, and then there exists $N_\ve\in \N $ such that $\noop{\pi_{\rh_k,\la_k}(F_\ve)\circ V_{k }\circ M_{I_k}}<\ve$ for $k> N_\ve$. Hence $\noop{\pi_{\rh_k,\la_k}(a)\circ V_{k }\circ M_{I_k}}<2\ve$ for $k>N_\ve $.
 \end{proof}
 
\subsection{Convergence in operator norm: $\om\ne 0$}$ $
  
In this section, we consider sequences $(\la_k)_k$ and $(\rh_k)_k$ satisfying 
\begin{align*}\label{}
 \limk \la_k= 0 \text{ and} 
 \limk \om_k= \limk \la_k\rh_k= \om\ne 0.
 \end{align*}
 Recall the equivalent representations $\ta^{\pm}_{\mu, \nu}$ given in \eqref{ta mu nu}. We have the following lemma. 

 \begin{lemma}\label{difference operator lak rhk ta} 
 Let $(\pi_{\rh_k,\la_k} )_k$ be  a properly converging sequence in $\wh G $ 
such that $\la_k=\ve  \vert{\la_k }\vert, k\in\N$, $\limk \la_k= 0 $ and that 
$\limk \om_k:= \limk \la_k\rh_k= \om\in\R^* $. Take a sequence $(R_k)_k $ in 
$\R_+ $ such that  $\limk R_k= \iy $ and 
\begin{eqnarray}\label{Rk conditions 2- nonzero om}
  \limk R_k\vert \la_k\vert = 0,\ \limk ({R_k^2}{\vert \la_k}\vert )= \iy, 
 \end{eqnarray}
equivalently, $\limk \frac{\rh_k}{R_k^2}= 0$. Then for any $F\in \l1 G$ with $\hat F^{3,4}\in C_c(\R^4) $, we have  
that 
\begin{eqnarray*}\label{limit -om+} 
 (a)& &
 \noop{\pi_{\rh_k,\la_k}(F)\circ V_{k}\circ M_{\{\geq \val{\la_k}R_k\}}- V_k\circ 
\ta^{+}_{\ve\om_k,-\ve}(F) \circ M_{\{\geq \val{\la_k}R_k\}} }\leq C\big(\frac{ 
\vert \om_k \vert }{\val{R_k^2\la_k}}+ R_k^{-1/2}\big),
\\
 (b) & &
  \noop{\pi_{\rh_k,\la_k}(F)\circ V_{k}\circ M_{\{\leq -\val{\la_k}R_k\}}- V_k\circ 
\ta^{-}_{-\ve\om_k,\ve}(F) \circ M_{\{\leq -\val{\la_k}R_k\}} }\leq 
C\big(\frac{ 
\vert \om_k \vert }{\val{R_k^2\la_k}}+ R_k^{-1/2}\big),
\end{eqnarray*}
 for some constant $C$ depending on $F $.
 \end{lemma}

\begin{proof} 
Take any $F\in \l1 G$ such that $\hat F^{3,4}\in C_c(\R^4) $.  Then $\hat 
F^{3,4}(t,s,x,\la)=0 $ for any $t,s,x\in\R $ such that $ \vert{t }\vert + \vert{s }\vert +\vert x\vert + \vert{\la}\vert \geq D $ for some $D>0 $. Take $C_1> 0 $ such 
that $ e^D C_1< 1$. 

Let $s, t, x\in \R$ such that $\val t\leq D, \val x\leq D$ and $\val s\geq C_1\val{ R_k\la_k}$. We have that
 \begin{eqnarray}\label{ln(s+small)}
 \ln(\vert  e^ts-\la_k x\vert )-\ln(\vert e^t s\vert )+{\la_k x}( e^ts)\inv  
&=& -\la_k x \int_{0}^{1}\big(\frac{1}{ e^ts-\la_k x v }-\frac{1}{ e^ts 
}\big)dv\\
\nn  &= &
\frac{-(\la_k x)^2}{e^t s} \int_{0}^{1}\frac{v dv}{ e^ts-\la_k x v } \\ 
 \nn  &=: &   \la_k{c_k(t,s, x)},
 \end{eqnarray}
where
\begin{eqnarray*}
 c_k(t,s,x)= \frac{-\la_k x^2}{e^t s} \int_{0}^{1}\frac{ vdv}{ e^ts-\la_k x v }.
 \end{eqnarray*}
Hence for $k $ large enough and for some new constant $C= \frac{D^2}{2 C_1^2}$, we have:
\begin{eqnarray*}
 \vert \ln(\vert e^ts-\la_k x\vert )-\ln(\vert e^t s\vert )+{\la_k x}\vert 
e^ts\vert \inv  \vert &=&
 \vert  \la_k c_k(t,s,x)  \vert \\ 
 \nn  &\leq & C\frac{\la_k^2}{s^2} \leq C\frac{1}{  R_k^2},
 \end{eqnarray*}
for all our $s, x, t $ in $\R$. Whence
\begin{eqnarray}\label{one minus e ick} 
 \vert e^{i\rh_k\la_kc_k(t,s,x)}-1\vert &\leq &\val{\rh_k\la_kc_k(t,s,x)}\\
\nn  &\leq &
\frac{C\val{\rh_k\la_k}}{\val{\la_k R_k^2}}.
\end{eqnarray}
By (\ref{vkstar circ Vk})  we have that 
\begin{eqnarray*}
 V_{k}^*(\pi_{\rh_k,\la_k}(F)(V_{k}(\et)))(s)
&= &
 \frac{e^{-i\rh_k\ln(\val s)}\val{s}^{1/2}}{ \vert{\la_k }\vert 
}{}\int_\R \int_\R  e^{-i\rh_k t}{e^{t/2}}\hat 
F^{3,4}(t,\frac{-\ve}{\la_k}( x-{e^ts}),\frac{-\ve}{2}( 
x+{e^ts}),\la_k)dt\\
 \nn & & \frac{e^{i\rh_k\ln(\val{x})}}{\val{ x 
}^{1/2}}\et( x)dx \ \quad \text{for } \, \,  \et\in L^2(\R, \frac{dx}{\val x}).
 \end{eqnarray*}
It follows for our constant $C_1 $ that for any  
  $\et\in L^2(\R,\frac{dx}{ \vert{x }\vert}) $ with $\et(x)= 0 $,  
$\val x\leq  R_k \vert{\la_k }\vert $ and $k $ large enough, 
 \begin{eqnarray*}\label{}
\hat F^{3,4}(t,\frac{-\ve}{\la_k}( x-{e^ts}),\frac{-\ve}{2}( 
x+{e^ts}),\la_k)\et( x)=0
 \end{eqnarray*}
 for every $t,x,\la_k\in\R$ and $\vert{s }\vert< C_1 R_k \vert{\la_k }\vert $. 
For $\val s\geq C_1R_k\val{\la_k} $, we have that 
 \begin{eqnarray*}
 \nn  & &
\big| V_{k}^*(\pi_{\rh_k,\la_k}(F)(V_{k }(\et)))(s)\\
 \nn  & &- \int_\R \int_\R   \hat F^{3,4}(t,x,\frac{\la_k}{2}x-\ve{e^ts}{},\la_k) 
\frac{\val{e^{t}s}^{1/2}e^{-i\ve\rh_k{\la_k x}( e^ts)\inv }}{\val{-\ve\la_k x+ 
e^t  s }^{1/2}}\et(- \ve\la_k x+ e^t s )dxdt \big|\\
\nn  &= & \big|\int_\R \int_\R \hat 
F^{3,4}(t,x,\frac{\la_k}{2}x-\ve{e^ts}{},\la_k)\val{e^{t}s}^{1/2}\\
 \nn & &
\big(\frac{e^{-i\ve\rh_k{\la_k x}(\vert e^ts\vert )\inv + 
i\ve\rh_k\la_k{c_k(t,s,\ve x)}}-e^{-i\ve\rh_k{\la_k x}(\vert e^ts\vert )\inv 
}}{\val{-\ve\la_k x+  e^t  s }^{1/2}}\big) \et(-\ve\la_k x+e^t s ) dxdt\big| \\
\nn  &\leq &
\int_\R \int_\R  \vert \hat F^{3,4}(t,x,\frac{\la_k}{2}x-\ve{e^ts},\la_k)\vert 
\frac{\vert e^{i\ve\rh_k\la_k{c_k(t,s,\ve x)}}-1\vert}{\val{-\ve\la_k x+e^t  
s}^{1/2}}\val{e^{t}s}^{1/2}\vert \et(-\ve\la_k x+e^t s )\vert dxdt\\
\nn  &\overset{(\ref{one minus e ick})}{\leq}& 
\frac{C \vert \rh_k \la_k\vert }{\val{\la_k R_k^2}}\int_\R \int_\R  \vert \hat 
F^{3,4}(t,x,\frac{\la_k}{2}x-\ve{e^ts},\la_k)\vert \frac{1 }{\val{-\ve\la_k x+  
e^t  s }^{1/2}}\val{e^{t}s}^{1/2}{\vert \et(-\ve\la_k x+e^t s )\vert} dxdt.
 \end{eqnarray*}
Therefore, for any $\et\in L^2(\R,\frac{dx}{ \vert{x }\vert}) $ with $\et(x)= 0 
$ for $\val x\leq  R_k \vert{\la_k }\vert $ and $k $ large enough, we deduce 
that 
\begin{eqnarray*}
 \nn  & &
\int_{\{\val s\geq C_1 R_k\val{\la_k}\}}\val{ 
V_{k}^*(\pi_{\rh_k,\la_k}(F)(V_{k}(\et)))(s)\\
 \nn  & &- \int_\R \int_\R  \hat F^{3,4}(t,x,\frac{\la_k}{2}x-\ve{e^ts},\la_k) 
\frac{\val{e^{t}s}^{1/2}e^{-i\ve\rh_k{\la_k x}( e^ts)\inv }}{\val{-\ve\la_k x+ 
e^t  s }^{1/2}}\et(-\ve\la_k x+e^t s )dxdt}^2\frac{ds}{\vert s\vert }\\
\nn  &\leq &
\int_{\{\val s\geq C_1 R_k\val{\la_k}\}}{\Big(\frac{C  \vert{\rh_k 
\la_k}\vert}{\val{\la_k R_k^2}}\int_\R \int_\R  \vert \hat 
F^{3,4}(t,x,\frac{\la_k}{2}x-{{\ve e^ts}},\la_k)\vert \frac{1 
}{\val{-{\ve\la_k 
x}+  e^t  s}^{1/2}}}\val{e^{t}s}^{1/2}\vert {\et(-\ve\la_k x+e^t s )} \vert 
dxdt\Big)^2\frac{ds}{\val s}\\
\nn  &\leq &
(\frac{C \vert \rh_k \la_k\vert }{\val{\la_k R_k^2}})^2\int_{\{\val s\geq C_1 
R_k\val{\la_k}\}}{\int_\R \int_\R  \vert e^{t/2}\hat F^{3,4}(t,x, 
\frac{\la_k}{2}x-\ve{e^ts},\la_k)\vert \frac{1 }{\val{-\ve\la_k x+  e^t  s 
}}}\vert \et(-\ve\la_k x+e^t s )\vert^2 dxdtds\\
\nn  & &
\cdot \sup_{\val s\geq C_1 R_k\val{\la_k}}\int_\R\int_\R e^{r/2}\vert \hat 
F^{3,4}(r,u,\la_k u-\ve e^rs,\la_k)\vert dudr.
 \end{eqnarray*}
Choose a positive real-valued function $\va\in C_c^{\iy}(\R, \R_+)$ such that
\begin{eqnarray*}
 \val{e^{t/2}\hat F^{3,4}(t,x,u,\la_k)}\leq \va(t)\va(x)\va(u) \, \text{ for } 
\, k\in\N \, \text{ and} \ u, t, x\in \R.
 \end{eqnarray*} 
Then we obtain the inequality: 
\begin{eqnarray}\label{difference one}
  & &
\int_{\{\val s\geq C_1 R_k\val{\la_k}\}}\val{V_{k}^*(\pi_{\rh_k,\la_k}(F)(V_{k}(\et)))(s)\\
 \nn  & &- \int_\R\int_\R  \hat F^{3,4}(t,x,\frac{\la_k}{2}x-\ve{e^ts},\la_k) 
\frac{\val{e^{t}s}^{1/2}e^{-i\ve\rh_k{\la_k x}( e^ts )\inv }}{\val{-\ve\la_k x+ 
e^t s}^{1/2}}\et(-\ve\la_k x+e^t s )dxdt}^2\frac{ds}{\vert s\vert }\\
\nn  &\leq &
(\frac{C \vert \rh_k \la_k\vert }{\val{\la_k  }
R_k^2})^2
\int_\R \int_\R \int_\R e^{-t} \va(t)\va(x) \frac{1 }{\val{  s }}\vert \et(e^t s)\vert^2 dxdtds\\
\nn  &\leq &
(\frac{C \vert \rh_k \la_k\vert }{\val{\la_k R_k^2}})^2\no\et^2_2,
 \end{eqnarray}
for some new constant $C> 0 $ depending on $F $.

 Now
 \begin{eqnarray*}
  e^{t} s-\ve\la_k x= e^{t} s\al_k(t,s,x)=e^{t+\ln(\al_k(t,s,x))}s,
 \end{eqnarray*}
where
\begin{eqnarray*}
 \al_k(t,s,x)= 1- \frac{\ve\la_k x}{ e^{t}s}.
 \end{eqnarray*}
If $\val {s}\geq C_1 \val{\la_k}R_k $ for the above constant $C_1$,  $k$ large enough 
and for all $\val t\leq D, \val x\leq D$, it follows that 
\begin{eqnarray}\label{al k estimated}
\nn \vert \frac{\la_kx}{e^ts}\vert &\leq& \frac{e^D \vert \la_k \vert D}{C_1 \vert \la_k 
\vert R_k }= \frac{D}{C_1  R_k },\\ 
\vert \ln(\al_k(t,s,x))\vert &\leq& C\frac{1}{R_k}, \, \text{and} \\
\nn \val{\partial_t \ln(\al_k(t,s,x))}&\leq& C\frac{1}{R_k},
 \end{eqnarray}
for some (new) constant $C> 1$.

Let \begin{eqnarray*} 
 r= \ta_k(t) &:= &t+\ln(\al_k(t,s,x)), \\
 \nn \ta\inv_k(r)= \mu_k(r) & =& t \quad \ \text{for} \quad \val t\leq D, \val x\leq D, \val{s}\geq  C_1 \val{\la_k R_k},
\end{eqnarray*}
we have that $\val{\mu_k(r)-r}= \val{\ln(\al_k(t,s,x))}\leq C\frac{1}{R_k}$. Then
\begin{eqnarray*} 
  & &\int_{\R} \int_\R  \hat F^{3,4}(t,x,\frac{\la_k}{2}x-\ve{ e^ts}{},\la_k)
\frac{\val{e^{t}s}^{1/2}e^{-i\ve\rh_k{\la_k x}( e^ts )\inv }}{\val{-\ve\la_k x+ 
 e^t  s }^{1/2}}\et(-\ve\la_k x+e^t s )dxdt\\ 
&= &
\int_{\R}  \int_\R  \hat F^{3,4}(\mu_k(r),x, e^{r}s,\la_k) 
\frac{\val{e^{\mu_k(r)}s}^{1/2}e^{-i\ve\rh_k{\la_k x}( 
e^{\mu_k(r)}s)\inv}}{\val{e^{r}  s }^{1/2}}\et(e^{r} s )dx\mu'_k(r)dr.
\end{eqnarray*}

Now by (\ref{al k estimated}), for $k $ large enough, we have
\begin{eqnarray*}
   & &\vert \hat F^{3,4}(\mu_k(r),x,{ 
e^rs},\la_k){\val{e^{\mu_k(r)}s}^{1/2}e^{-i\ve\rh_k{\la_k x}( e^{\mu_k(r)}s 
)\inv}}\mu'_k(r)\\  
 & & \quad - \hat F^{3,4}(r,x,{e^rs}{},0){\val{e^{r}s}^{1/2}e^{-i\ve\rh_k{\la_k 
x}( e^rs)\inv }}\vert \\
  &\leq &
\frac{C}{R_k}\va(r)\va(x)\val s^{1/2},
\end{eqnarray*}
where $\va\in C_c^{\iy}(\R, \R_+)$ is given before \eqref{difference one}.
Therefore for $k $ large enough and any $\et\in \l2{ \R, \frac{ds}{\val s}} $ 
with $\|\et\|_2\leq 1 $ and $\et(u)= 0 $ for $\vert  u\vert \leq R_k 
\vert{\la_k}\vert $, we have that
\begin{eqnarray}\label{first approx}
  &  &
\int_{\R}\big| \int_{\R} \int_\R  \hat 
F^{3,4}(t,x,\frac{\la_k}{2}x-\ve{e^ts}{},\la_k) 
\frac{\val{e^{t}s}^{1/2}e^{-i\ve\rh_k{\la_k x}( e^ts )\inv }}{\val{-\ve\la_k x+ 
e^t  s }^{1/2}}\et(-\ve\la_k x+e^t s )dxdt\\ 
\nn  & &- \int_{\R}  \int_\R  \hat F^{3,4}(r,x,-{ 
e^rs}{},0){e^{-i\ve\rh_k{\la_k 
x}( e^{r}s )\inv }}{}\et( e^r s )dxdr{\big|}^{2}\frac{ds}{\val s }\\
  \nn&= &
  \int_\R\big| \int_{\R}  \int_\R  \hat F^{3,4}(\mu_k(r),x,-e^r s,\la_k) 
\frac{\val{e^{\mu_k(r)}s}^{1/2}e^{-i\ve\rh_k{\la_k x}( e^{\mu_k(r)}s 
)\inv}}{\val{e^r  s }^{1/2}}\et( e^r s )dx\mu'_k(r)dr\\
\nn  & &- \int_{\R}  \int_\R  \hat F^{3,4}(r,x,{ 
-e^rs}{},0){e^{-i\ve\rh_k{\la_k 
x}( e^{r}s )\inv }}{}\et( e^r s)dxdr\big|^{2}\frac{ds}{\val s}\\
\nn  &  \leq&
\frac{C}{R_k},
 \end{eqnarray}
for a new constant $C> 0$ depending on $F$.

Let us recall that for $\et\in L^2(\R_\si, \frac{dx}{\val x}) $ (see \eqref{ta mu nu}),
  \begin{eqnarray*}\label{ta muu nuu }
 \ta_{\mu,\nu}^{\si}(F)\et(u)=
\int_\R \hat F^{2,3,4}(t,\mu \si(e^{t}u)\inv ,\si\nu ( e^tu),0)\et(e^tu)dt, \ u\in \R_\si. 
 \end{eqnarray*}
Hence  for any $\et\in \l2{\R_+, \frac{ds}{\val s}} $, we have that
\begin{eqnarray*} 
  & &\Vert{V_{k}^*\circ\pi_{\rh_k,\la_k}(F)\circ V_{k}\circ M_{\{\geq\val 
{R_k\la_k}\}}(\et)-\ta^{+}_{\ve\om_k,- \ve}(F)\circ M_{\{\geq 
\val{\la_k}R_k\}}(\et)}\Vert _2\\
  \nn  &= & 
  \Big(\int_\R \vert V_{k}^*\circ\pi_{\rh_k,\la_k}(F)\circ V_{k}\circ 
M_{\{\geq\val {R_k\la_k}\}}(\et)(s)\\
\nn  & &- \int_\R \hat F^{2,3,4}(t,\ve\om_k s\inv e^{-t}, -s e^t,0)M_{\{\geq 
\val{\la_k}R_k\}}(\et)( e^ts)dt \vert^{2}\frac{ds}{\val s}\Big)^{1/2}\\
  \nn  &= &
  \Big(\int_\R \vert V_{k}^*\circ\pi_{\rh_k,\la_k}(F)\circ V_{k}\circ 
M_{\{\geq\val {R_k\la_k}\}}(\et)(s)\\
\nn  & &- \int_\R \int_\R\hat F^{3,4}(t,x,-{ e^ts},0) {\val{e^{t}s}^{1/2}e^{-i 
\ve\rh_k{\la_k x}( e^ts )\inv }}M_{\{\geq \val{\la_k}R_k\}}(\et)( e^ts)dtdx 
\vert ^{2}\frac{ds}{\val s}\Big)^{1/2}\\
 \nn  &\overset{(\ref{first approx})}\leq &
  (\frac{C }{R_k})^{1/2} \no\et_2+ \Big(\int_\R \vert 
V_{k}^*\circ\pi_{\rh_k,\la_k}(F)\circ V_{k}\circ M_{\{\geq\val 
{R_k\la_k}\}}(\et)(s)\\
\nn  & &- \int_\R \int_\R  \hat F^{3,4}(t,x,\frac{\la_k}{2}x-\ve{ e^ts}{}),0) 
\frac{\val{e^{t}s}^{1/2}e^{-i\ve\rh_k{\la_k x}( e^ts )\inv }}{\val{-\ve\la_k x+ 
e^t  s }^{1/2}}M_{\{\geq R_k\val{\la_k}\}}(\et)(-\ve\la_k x+ e^t s)dxdt\vert 
^2\frac{ds}{\vert s\vert }\Big)^{1/2}\\
\nn  &\leq &
  (\frac{C }{R_k})^{1/2} \no\et_2+ \Big(\int_\R \vert \int_\R \int_\R  \hat 
F^{3,4}(t,x,\frac{\la_k}{2}x-\ve{ e^ts}{},0)\\
 \nn & &
\frac{\val{e^{t}s}^{1/2}e^{-i\ve\rh_k{\la_k x} (e^ts)\inv + 
i\rh_k\la_k{c_k(t,s,\ve x)}}}{\val{-\ve\la_k x+ e^t  s}^{1/2}}M_{\{\geq\val 
{R_k\la_k}\}}\et(- \ve\la_k x+  e^t s )dxdt\\
\nn  & &- \int_\R \int_\R  \hat F^{3,4}(t,x,\frac{\la_k}{2}x-\ve{ e^ts}{}),0) 
\frac{\val{e^{t}s}^{1/2}e^{-i\ve\rh_k{\la_k x}( e^ts )\inv }}{\val{-\ve\la_k x+ 
e^t  s }^{1/2}}M_{\{\geq R_k\val{\la_k}\}}\et(-\ve\la_k x+ e^t 
s)dxdt\vert^2\frac{ds}{\vert s\vert }\Big)^{1/2}\\
\nn  &\leq &
  (\frac{C }{R_k})^{1/2} \no\et_2+ \Big(\int_\R \vert \int_\R \int_\R  \hat 
F^{3,4}(t,x,\frac{\la_k}{2}x-\ve{ e^ts}{},0)\\
 \nn & & 
 \Big\vert e^{i\ve\rh_k\la_k{c_k(t,s,\ve x)}}-1\Big\vert 
\frac{\val{e^{t}s}^{1/2}e^{-i\ve\rh_k{\la_k x} (e^ts)\inv}}{\val{-\ve\la_k x+ 
e^t  s }^{1/2}} M_{\{\geq\val{R_k\la_k}\}}\et(-\ve \la_k x+  e^t s )dxdt 
\vert^2\frac{ds}{\vert s\vert }\Big)^{1/2}\\
\nn  &\overset{(\ref{one minus e ick})}\leq &
C(\frac{ \vert \rh_k \la_k\vert }{\val{\la_k 
R_k^2}}+R_k^{-1/2})\no\et_2,
\end{eqnarray*}
for a constant $C> 0 $.  Then for any  $s\in\R $, it follows that
\begin{eqnarray*}
\nn  & &
C(\frac{ \vert \om_k \vert }{\val{R_k^2\la_k}}+ R_k^{-1/2})\\ 
\nn  &{\geq} &  \noop{ \pi^{\al}_{\rh_k,\la_k}(F)\circ V_{k}\circ M_{\{\geq 
\val{\la_k}R_k\}}- V_k\circ (\ta^{+}_{\ve\om_k,-\ve})^{\al}(F) \circ M_{\{\geq 
\val{\la_k}R_k\}}}\\
\nn  & \overset{(\ref{pilarh s}), (\ref{equi - + s})}{=} &
\noop{S\circ  \pi_{\rh_k,\la_k}(F)\circ  S\circ V_{k}\circ M_{\{\geq 
\val{\la_k}R_k\}}- V_k\circ \ta^{+}_{-\ve\om_k,\ve}(F) \circ M_{\{\geq \val{\la_k}R_k\}} }  \\
\nn  &\overset{(\ref{ScommVk}), (\ref{tamunu circ S})}=&
\noop{  \pi_{\rh_k,\la_k}(F)\circ V_{k}\circ M_{\{\leq -\val{\la_k}R_k\}}- V_k\circ 
\ta^{-}_{-\ve\om_k,\ve}(F) \circ M_{\{\leq -\val{\la_k}R_k\}} }.
\end{eqnarray*}
\end{proof}

Let $(\pi_{\rh_k,\la_k} )_k$ be a properly converging sequence in $\wh G$ such that $\limk \la_k=  0$ and $\limk \la_k\rh_k= \om \in \R^*$. We recall (Proposition~\ref{limtwothree}) that the limit set $L $ of the sequence $(O_{\rh_k,\la_k})_k $ is the two points set
\begin{eqnarray*}
 L= \{O_{-\om,1,0}, O_{\om,-1,0}\}.
 \end{eqnarray*}
For $k\in \N$, let $\la_k= \ve \val{\la_k}$ and $(R_k)_k$ be a sequence in 
$\R_+$ satisfying the conditions given in \eqref{Rk conditions 2- nonzero om}, 
namely, $ \limk R_k\vert \la_k\vert = 0$ and $\limk ({R_k^2}{\vert \la_k}\vert )= 
\iy$.  We define $ \si_{k}^\om$ by 
\begin{align}\label{control om}
\si_{k}^\om(\phi|_L):= V_{k}\circ \Big(\phi(\ta^{+}_{\ve\om_k,-\ve})\circ M_{\{\geq \val{\la_k}R_k\}}\oplus 
{\phi(\ta^{-}_{-\ve\om_k,\ve})}\circ M_{\{\leq-\val{\la_k}R_k\}}\Big) \circ V_k^*, \quad \phi \in l^\iy(\hat G),
 \end{align}
where $\ta^{\si}_{\mu, \nu}$ acting on $L^2(\R_{\si}, \frac{du}{|u|})$ is given in \eqref{ta mu nu def}.
 
 \begin{theorem}\label{lim rhklak is om non0}
Suppose that $ \limk 
\rh_k\la_k = \om\in\R^*$ and $\la_k= \ve \val{\la_k}$ for $k\in\N$. Then
\begin{eqnarray*}
 \limk \noop{\pi_{\rh_k,\la_k}(a)-\si_{k}^\om(\hat a\res{L})}= 0
 \end{eqnarray*}
for every $ a\in C^*(G)$.
 \end{theorem}
 
\begin{proof} 
Apply Proposition~\ref{norm conv} and  Lemma~\ref{difference operator lak rhk ta}.
Let $a\in C^*(G) $ and $\ve> 0$. Choose $F\in L^1(G) $ with $\hat F^{3,4}\in C_c(\R^4)$ such that $\no{a-F}_{C^*(G)}< \ve $. Then by Lemma~\ref{difference operator lak rhk ta}, there exists an $N_\ve\in\N $ such that $\noop{\pi_{\rh_k,\la_k}(F)-\si_{k}^\om(\hat F\res L)}< \ve$ for $k\geq N_\ve $. Hence 
\begin{align*}
\noop{\pi_{\rh_k,\la_k}(a)-\si_{k}^\om(\hat a\res{L})}< 2\ve \quad \text{for } \, k\geq N_\ve. 
\end{align*}
\end{proof}  

\subsection{Convergence in operator norm: $\om=0 $}\label{om is  0}$ $ 

Suppose now that
\begin{eqnarray*}
\limk \la_k= 0= \limk \la_k\rh_k.
 \end{eqnarray*}
We can again assume that $\la_k= \ve \val{\la_k}$ for $k\in\N $. By Proposition~\ref{limtwothree}, the limit set $L $ of the sequence of 
representations $(\pi_{\rh_k,\la_k})_k$ is equal to $\GA_1\cup\GA_0 $. We first show that a similar convergence as the one in Lemma~\ref{limk de=0} holds when $\om= 0$, but with a slightly different sequence $(R_k)_k$ in $\R_+$.

\begin{lemma}\label{limk de=0 and more}
Suppose that $\limk \la_k= 0 $. Take a sequence $(R_k)_k $ in $\R_+ $ such that  
\begin{eqnarray}\label{Rk conditions 1- zero om}
\limk R_k= +\iy, \ \limk R_k^{2}\vert \la_k\vert = 0.
 \end{eqnarray}
 Then we have that 
  \begin{eqnarray}\label{small to 0}
 \limk \noop{\pi_{\rh_k,\la_k}(a)\circ V_{k }\circ M_{\{ \vert{ }\vert\leq \vert 
R_k\la_k\vert\} }}= 0,\quad a\in C^*(G).
 \end{eqnarray}
 \end{lemma}
 
\begin{proof} 
It suffices to consider $F\in \l1 G $ such that $\hat F^{3,4}\in C_c^\iy(\R^4)$. There exists $M> 0 $ such that 
\begin{eqnarray*}
\hat F^{3,4}({ t},x,y,\la)= 0
\end{eqnarray*}
whenever $ \vert{t }\vert+ \vert{x }\vert+ \vert{y }\vert+ \vert{\la }\vert\geq M $. We have  that
\begin{eqnarray}\label{vk star pik Vk}
V_{k}^*(\pi_{\rh_k,\la_k}(F)(V_{k }(\et)))(s)
 \nn  &= &
 \frac{e^{-i\rh_k\ln(\val s)}\val s^{1/2}}{\val{\la_k}^{1/2}}\int_\R 
\big(\int_\R e^{t/2} e^{-i\rh_k t}\hat F^{3,4}(t,\ve\frac{e^ts}{\la_k}- 
x,-\frac{\la_k}{2}(x+\ve\frac{e^ts}{\la_k}), \la_k)dt\big)\\
 \nn  & &
\frac{e^{i\rh_k\ln(\val{\la_kx})}}{\val x^{1/2}}\et(\ve\la_k x)dx\\
 &= &
 \frac{e^{-i\rh_k\ln(\val s)}\val s^{1/2}}{\val{\la_k}^{1/2}}\int_\R 
\big(\int_\R e^{t/2} e^{-i\rh_k t}\hat 
F^{3,4}(t,\ve\frac{e^ts}{\la_k}-\frac{\ve 
x}{\la_k},-\frac{1}{2}(\ve x+\ve{e^ts}{}),\la_k)dt\big)\\
 \nn  & &
{e^{i\rh_k\ln(\val{x})}\val x^{1/2}}{}\et( x)\frac {dx}{\val x}.
 \end{eqnarray}
Hence for $k $ large enough, $\vert{t}\vert\leq  M$, $\vert{ x}\vert\leq R_k \vert{ \la_k }\vert$, and $ \vert{s}\vert\geq C R_k  \vert{\la_k}\vert $ for some constant $C> 3 e^M$, we have that $\vert{\ve\frac{e^ts}{\la_k}-\frac{\ve x}{\la_k} }\vert> M $ and so $\hat F^{3,4}(t,\ve\frac{e^ts}{\la_k}-\frac{\ve x}{\la_k},-\frac{1}{2}(\ve x+\ve{e^ts}{}), { \la_k})=0 $.  Thus,
\begin{eqnarray*}
 M_{\{ \vert{ }\vert\geq C R_k \val{\la_k}\}}\circ 
V_{k}^*\circ \pi_{\rh_k,\la_k}(F)\circ V_{k }\circ  M_{\{ \vert{ }\vert\leq R_k 
\vert{\la_k }\vert\}}= 0 
 \end{eqnarray*}
 for $k $ large enough. Choose an even function $\va:\R \to \R_+ $ in $C_c(\R) $ with compact support such that
\begin{eqnarray*}
 \val{\hat F^{3,4}(t,x,y,\la)}\leq \va(t)\va(x)\va(y),\quad t, x, y, \la \in \R.
 \end{eqnarray*}
Now, by Young's inequality, 
\begin{eqnarray*}
 \nn  & &
\noop{M_{ \{\vert{ }\vert\leq C R_k \val{\la_k}\}}\circ V_{k}^*\circ  
\pi_{\rh_k,\la_k}  (F)\circ V_{k }\circ  M_{ \{\vert{ }\vert\leq R_k 
\vert{\la_k 
}\vert\} }}\\
\nn  &\leq &
\sup_{\val x\leq R_k \vert{\la_k }\vert }\int_{ \vert{s }\vert\leq C R_k 
\vert{\la_k }\vert}\frac{\val {x s}^{1/2}}{\val{\la_k}^{1/2}}\big(\int_\R 
e^{t/2} \vert \hat F^{3,4}(t,\ve\frac{e^ts}{\la_k}-\frac{\ve 
x}{\la_k},-\frac{1}{2}(\ve x+\ve{e^ts}{}),\la_k)\vert dt\big) \frac{ds}{\val 
s}\\
\nn  & &
+ \sup_{\val s\leq C R_k \vert{\la_k }\vert }\int_{ \vert{x }\vert\leq  R_k 
\vert{\la_k }\vert}\frac{\val {x s}^{1/2}}{\val{\la_k}^{1/2}}\big(\int_\R 
e^{t/2} \vert \hat F^{3,4}(t,\ve\frac{e^ts}{\la_k}-\frac{\ve 
x}{\la_k},-\frac{1}{2}(\ve x+\ve{e^ts}{}),\la_k)\vert dt\big) \frac{dx}{\val 
x}\\
\nn  &\leq &
\sup_{\val x\leq  R_k \vert{\la_k }\vert }\int_{ \vert{s }\vert\leq C R_k 
\vert{\la_k }\vert}\frac{\val x^{1/2}}{\val{\la_k s}^{1/2}}\big(\int_\R e^{t/2} 
\va(t)\va({\frac{ e^ts}{\la_k}-\frac{ 
x}{\la_k}})\va(\frac{1}{2}(x+{e^ts}{}))dt\big)ds\\
\nn  & &
+ \sup_{\val s\leq C R_k \vert{\la_k }\vert }\int_{ \vert{x }\vert\leq R_k 
\vert{\la_k }\vert}\frac{\val s^{1/2}}{\val{\la_k x}^{1/2}}\big(\int_\R e^{t/2} 
\va(t)\va({\frac{ e^ts}{\la_k}-\frac{ x}{\la_k}})\va(\frac{1}{2}( 
x+{e^ts}{}))dt\big)dx\\
\nn  &\leq &
R_k^{1/2}\int_{ \vert{s }\vert\leq C R_k \vert{\la_k 
}\vert}\frac{1}{\val{s}^{1/2}}\big(\int_\R e^{t/2} 
\va(t)\va(\frac{e^ts}{\la_k}-\frac{x}{\la_k})\va(\frac{1}{2}( x+{e^ts}{})) 
dt\big)ds\\
\nn  & &
+ C^{1/2} R_k^{1/2}\int_{ \vert{x }\vert\leq R_k 
\vert{\la_k}\vert}\frac{1}{\val{x}^{1/2}}\big(\int_\R e^{t/2} 
\va(t)\va(\frac{e^ts}{\la_k}-\frac{x}{\la_k})\va(\frac{1}{2}( x+{e^ts}{})) 
dt\big)dx\\
\nn  &\leq &
R_k^{1/2}\no\va_\iy^2\int_{ \vert{s }\vert\leq C R_k 
\vert{\la_k}\vert}\frac{1}{\val{ s}^{1/2}}ds\int_\R e^{t/2} \va(t) dt\\
\nn  & &
+ C^{1/2} R_k^{1/2}\no\va_\iy^2\int_{ \vert{x }\vert\leq R_k 
\vert{\la_k}\vert}\frac{1}{\val{ x}^{1/2}}dx\int_\R e^{t/2} \va(t) dt\\
\nn  &= &
R_k^{1/2}\no\va_\iy^2{2C^{1/2}R_k^{1/2}\val{\la_k}^{1/2}}\int_\R e^{t/2} 
\va(t)dt\\
\nn  & &
+ C^{1/2}R_k^{1/2}\no\va_\iy^2{2R_k^{1/2}\val{\la_k}^{1/2}}\int_\R 
e^{t/2}\va(t) 
dt\\
\nn  &\leq & 
C' R_k \vert{\la_k }\vert^{1/2}
 \end{eqnarray*}
 for some constant $C'> 0$. 
\end{proof}

\begin{lemma}\label{dek muk}
Let $(\rh_k)_k$ be a real sequences with $\limk \rh_k=0 $. Then for any real sequence $(\la_k )_k$ we have that
\begin{eqnarray*}\label{}
 \nn\limk \noop{\pi_{\rh_k,\la_k}(a)- {\pi_{0,\la_k}}(a)}= 0, \quad a\in C^*(G).
 \end{eqnarray*}
 \end{lemma}
 
\begin{proof} 
For $k\in \N$, the identity
\begin{eqnarray*}\label{}
 \nn \pi_{\rh_k,\la_k}(F)- \pi_{0,\la_k}(F)= \int_{G} { 
(\chi_{\rh_k}(g)-1)}\pi_{0,\la_k}(g)F(g)dg, \, F\in C_c(G),  
 \end{eqnarray*}  
where $\ch_{\rh_k}(\exp {t T}\cdot h)= e^{-i\rh_k t}$ for $t\in\R$ and $h\in H$, the Heisenberg group, shows that
\begin{eqnarray*}\label{}
 \nn \noop{\pi_{\rh_k,\la_k}(F)-{\pi_{0,\la_k}}(F)}\leq \vert{\rh_k }\vert\no {F_1}_1,
 \end{eqnarray*}
where $F_1(t,x,y,z):= t F(t,x,y,z)$ and $g= (t,x,y,z)\in G$.
\end{proof}

 \begin{lemma}\label{difference operator lak rhk ta om0} 
Let $(\pi_{\rh_k,\la_k} )_k$ be a properly converging sequence in $\wh G $ such that $\limk \la_k= 0 $ and $\om_k:= \la_k\rh_k, \limk \om_k=0 $. Take a sequence $(R_k)_k $ in $\R_+ $ such that  $\limk R_k= \iy $ and 
\begin{eqnarray}\label{Rk conditions 2- zero om}
\limk R_k^2\la_k= 0, \ \limk \frac{{\om_k}}{R_k^2\val{\la_k}}= 0. 
 \end{eqnarray}
Let $F\in \l1G $ such that $\hat F^{3,4}\in C_c(\R^4) $. For $\ve\in\{+1, -1\} $, 
$\la_k= \ve \vert{\la_k }\vert $, we have that 
\begin{enumerate}\label{limit -om+}
\item[(a)] $\nn \noop{\pi_{\rh_k,\la_k}(F)\circ V_{k}\circ M_{\{\geq \val{\la_k}R_k\}}- 
V_k\circ \ta^{+}_{\ve\om_k,-\ve}(F) \circ M_{\{\geq \val{\la_k}R_k\}} }\leq 
C\big(\frac{ \vert \om_k \vert }{\val{R_k^2\la_k}}+ R_k^{-1/2}\big)$,
\item[(b)]  
 \nn
 $\noop{\pi_{\rh_k,\la_k}(F)\circ V_{k}\circ M_{\{\leq -\val{\la_k}R_k\}}- V_k\circ 
\ta^{-}_{-\ve\om_k,\ve}(F) \circ M_{\{\leq -\val{\la_k}R_k\}} }\leq 
C\big(\frac{ 
\vert \om_k \vert}{\val{R_k^2\la_k}}+R_k^{-1/2}\big)$,
 \end{enumerate}
for some constant $C $ depending on $F $.
\end{lemma}

\begin{proof} The statements (a) and (b) are proved in the same way as the corresponding ones in Lemma~\ref{difference operator lak rhk ta}. 
\end{proof} 

\begin{remark}
 We can take, for example, $R_k= \frac{1}{\val{\la_k}^{1/3}}, k\in\N$, if the sequence $(\rh_k)_k $ is bounded. Otherwise, for any $k\in \N$, let $R_k^2= \val{\rh_k}M_k$, for some $M_k\in \R$ such that $\limk M_k= \iy$ and $\limk M_k\om_k= 0 $. In the second case we have that 
\begin{eqnarray*}\label{}
 \nn \limk \frac{\om_k}{R_k^2\la_k}&=&
 \limk \frac{\rh_k}{R_k^2}\\
 \nn  &= &
\limk\frac{1}{M_k}= 0.
 \end{eqnarray*}
 \end{remark}
 
For the following arguments, we need to work with the multiplication operator $M_I$, where $I$ is a measurable subset of $\R$, which acts on the space $L^2(\R, d\mu)$ with any Borel measure $d\mu$ on $\R$.


\begin{definition}\label{Tk def}
Suppose that $\limk \la_k= 0 $, $\rh_k\ne 0, k\in\N$, and $\limk \val 
{\rh_k\la_k}= 0$. As before let $\om_k:=\la_k\rh_k$ for all $k\in\N $.
Choose two sequences $(Q_k)_k, (P_k)_k$ in $\R_+$ such that $\limk Q_k= \iy= 
\limk P_k $, $\limk \frac{Q_k}{P_k}= 0 $, $\limk \om_k P_k= 0$ and $R_k\leq 
\vert{\rh_k}\vert Q_k $ for all $k\in \N$, where the sequence $(R_k)_k$ is given as in \eqref{Rk conditions 2- zero om}, i.e. $\limk R_k= \iy $,  $\limk R_k^2\la_k= 0$ and $\limk \frac{{\om_k}}{R_k^2\val{\la_k}}= 0$. Define the subsets $J_k^{\pm}$ and $I_{k,j}^{\pm}$, for $k\in\N$ and $j\in\{ 1,2,3\}$, in $\R$ by
{\begin{eqnarray*}
\nn {J_{k}^+}&:=& \rbrack R_k\vert \la_k\vert, \vert \om_k\vert Q_k \rbrack,\ 
J_{k}^-:= \lbrack -\vert \om_k\vert  Q_k,-R_k\vert \la_k\vert  \lbrack,\\  
 \nn I_{k,1}^+&:= & \rbrack 0, \vert \om_k\vert  Q_k \rbrack,\ \ \ \ 
I_{k,1}^-:= \lbrack -\vert \om_k\vert  Q_k,0  \lbrack,\\
 \nn  I_{k,2}^+&:= &] \vert \om_k\vert  Q_k, \vert \om_k\vert  P_k ],\ 
\ \ \ I_{k,2}^-:= [-\vert \om_k\vert  P_k,-\vert  \om_k\vert  Q_k  [, \\  
I_{k,3}^+&:= & \rbrack \vert \om_k\vert P_k,\iy \lbrack,\ \ \ \ 
I_{k,3}^-:= \rbrack -\iy,-\vert \om_k\vert  P_k \lbrack.
 \end{eqnarray*}}
\end{definition}

\begin{lemma}\label{limk rhk lak=0}
Suppose that $\la_k= \ve \vert{\la_k }\vert$, $\rh_k\ne 0$, for $k\in \N$, 
$\limk \la_k= 0 $, and $\limk \val {\rh_k\la_k}= 0$. Take a sequence $(R_k)_k $ in 
$\R_+ $ and sequences $(Q_k)_k, (P_k)_k\subset \R_+$ satisfying the 
conditions in Definition~\ref{Tk def}. 
Then for any $a\in C^*(G) $, we have that 
 \begin{enumerate}\label{} 
\item[(a)] \begin{eqnarray*}
 \limk \noop{ \ta^+_{\om_k,-\ve }(a)-\ta^+_{\om_k,0}(a)\circ M_{J_{k}^+}\oplus 
\ta^+_{0,0 }(a)\circ M_{I_{k,2}^+}\oplus 
\ta^+_{0,-\ve\vert{\om_k  }\vert  }(a)\circ M_{I_{k,3}^+}}= 0,
\end{eqnarray*}
\item[(b)]  \begin{eqnarray*}
\limk \noop{ \ta^-_{-\om_k,\ve }(a)-\ta^-_{-\om_k,0}(a)\circ M_{J_{k}^-}\oplus 
\ta^-_{0,0 }(a)\circ M_{I_{k,2}^-}\oplus 
\ta^-_{0,\ve\vert{\om_k  }\vert  }(a)\circ M_{I_{k,3}^-}}= 0, 
 \end{eqnarray*}
 \end{enumerate}
where the representation $\ta^{\pm}_{\mu, \nu}$ is defined in \eqref{ta mu nu}.
 \end{lemma}
 
 \begin{proof} 
 First let $F\in \l1G $ such that $\hat F^{2,3,4}\in C_c^\iy(\R^4) $. Then 
there 
exists $\va\in C_c^\iy(\R) $ of non-negative values 
such that for all $t, t', x, x', y, y'\in \R$, 
\begin{eqnarray*}\label{}
 \nn  \vert{\hat F^{2,3,4}(t,x,y,0) }\vert &\leq &
 \va(t)\va(x)\va(y) \quad \text{and}\\
 \vert {\hat F^{2,3,4}(t,x,y,0)- {\hat F^{2,3,4}(t',x', y', 
0)}}\vert \nn  &\leq 
&
 \vert{t-t' }\vert\va(x)\va(y)+ \vert{x-x' }\vert\va(t)\va(y)+ 
\vert{y-y'}\vert\va(t)\va (x).
 \end{eqnarray*}
Then it follows that 
\begin{eqnarray*}\label{}
 & &
 \nn \vert  
(\ta^+_{\om_k,{-\ve}}(F)- \ta^+_{\om_k,0}(F))(M_{I_{k,1}^+}(\et))(u)\vert \\
&= &
 \vert \int_{I_{k,1}^+} (\hat F^{2,3,4}(t- \ln u, \om_k e^{-t},{ -\ve e^t},0)- 
\hat F^{2,3,4}(t- \ln u, \om_k e^{-t},0,0))M_{I_{k,1}^+}\et(e^t)dt\vert\\
&\leq &
 \int_{I_{k,1}^+} e^t \va(t-\ln(u))\va(\om_k e^{-t})\vert 1_{I^+_{k,1}}(e^t)\et(e^t)\vert dt\\
&\leq &
\vert\om_k\vert Q_k \int_{I_{k,1}^+} \va(t-\ln(u))\va(\om_k e^{-t})\vert 
\et(e^t)\vert dt.
 \end{eqnarray*} 
Hence
\begin{eqnarray*}\label{}
 \nn \no{(\ta^+_{\om_k,-\ve}(F)- \ta^+_{\om_k,0}(F))\circ M_{I_{k,1}^+}(\et)}^2 
&\leq &
 {\om_k }^2 Q_k^2\int_{\R_+} \Big(\int_{I_{k,1}^+} \va(t-\ln(u))\va(\om_k 
e^{-t})\vert \et(e^t)\vert dt\Big)^2\frac{du}{u}\\
\nn  
&\leq &
{\om_k }^2 Q_k^2\int_{\R_+} \Big(\int_{I_{k,1}^+}\va(t-\ln(u))dt 
\int_{I_{k,1}^+}\va(t-\ln(u))\va(\om_k e^{-t})^2\vert \et(e^t)\vert^2 
dt\Big)\frac{du}{u}\\
\nn  
&\leq &
{\om_k }^2 Q_k^2\no\va_1\no\va^2_\iy\int_{\R_+} 
\int_{I_{k,1}^+}\va(t-\ln(u))\vert \et(e^t)\vert^2 
dt\frac{du}{u}\\
\nn  &\leq &
{\om_k }^2 Q_k^2  \no\va^2_\iy \no\va_1^2 \no\et^2.
 \end{eqnarray*}
Since $0\leq \om_k Q_k\leq \om_k P_k = \vert{\la_k }\vert \vert \rh_k \vert P_k \to 0 $, it follows that
\begin{eqnarray*}\label{}
 \nn \limk\noop{\ta^+_{\om_k,-\ve}(F)\circ M_{I_{k,1}^+}- 
\ta^+_{\om_k,0}(F)\circ M_{I_{k,1}^+}}= 0.
 \end{eqnarray*}
Since for any $e^t\in I_{k,2}^+ $ we have that
\begin{eqnarray*}\label{}
 \nn \frac{1}{P_k} \leq \frac{ \vert{\om_k }\vert}{e^t}\leq \frac{1}{Q_k},
 \end{eqnarray*}
we get
\begin{eqnarray*}\label{}
 & &
 \nn \vert(\ta^+_{\om_k,{-\ve}}(F)- \ta^+_{0,0}(F))(M_{I_{k,2}^+}(\et))(u)\vert 
\\
&= &
 \vert \int_{I_{k,2}^+} (\hat F^{2,3,4}(t-\ln u,\om_k e^{-t},{ -\ve 
e^t},0)-\hat 
F^{2,3,4}(t-\ln u,0,0,0))\et(e^t) dt\vert\\
&\leq &
 \int_{I_{k,2}^+} 1_{I^+_{k,2}}(\om_k e^{-t})\vert \om_k e^{-t}\vert 
\va(t-\ln(u))\va(-\ve e^{t})\vert \et(e^t)\vert dt
+\int_{I_{k,2}^+}  1_{I^+_{k,2}}(-\ve e^{t})e^t 
\va(t-\ln(u))\va(\om_k e^{-t})\vert \et(e^t)\vert dt\\
&\leq &
(\frac{1}{Q_k}+ \vert \om_k\vert {P_k} )\no\va_\iy\int_{\R_+} 
\va(t-\ln(u))\vert 
\et(e^t)\vert dt.
 \end{eqnarray*} 
This relation implies that $\limk \noop{\ta^+_{\om_k,-\ve}(F)\circ M_{I_{k, 
2}^+}- \ta^+_{0,0}(F)\circ M_{I_{k, 2}^+}}= 0 $.

In the same way, we have that
\begin{eqnarray*}\label{}
 & &
 \nn \vert (\ta^+_{\om_k,-\ve}(F)- \ta^+_{0,-\ve\vert \om_k\vert}(F))(M_{I_{k,3}^+}(\et))(u)\vert \\
&= &
 \vert \int_{\R} (\hat F^{2,3,4}(t- \ln u, \om_k e^{-t},{ -\ve e^t},0)- \hat 
F^{2,3,4}(t- \ln u, 0, {-\ve \vert \om_k \vert e^t}, 0))M_{I_{k,3}^+}\et(e^t)dt \vert\\
&= &
 \vert \int_{I^+_{{k,3}}} (\hat F^{2,3,4}(t+\ln(\vert{\om_k }\vert ) -\ln 
u,\frac{1}{sign (\om_k) e^t}, { -\ve \vert{\om_k }\vert e^t},0)\\
& & 
- \hat F^{2,3,4}(t+\ln( \vert{\om_k }\vert )-\ln u, 0, {-\ve\vert{\om_k }\vert 
e^t}, 0)) \et( \vert{\om_k }\vert e^t) dt \vert\\
&\leq &
\int_{I_{k,3}^+} e^{-t}\va(t+ \ln( \vert{\om_k }\vert 
)- \ln(u)) \va(-\ve\vert \om_k\vert e^t)\vert \et( \vert{\om_k }\vert e^t)\vert dt\\
&\leq &
\frac{\no\va_\iy}{P_k} \int_{I_{k,3}^+} \va(t+ \ln(\vert{\om_k }\vert )- 
\ln(u))\vert \et( \vert{\om_k }\vert e^t)\vert dt.
 \end{eqnarray*} 
 Hence
 \begin{eqnarray*}\label{}
& &
\nn \no{(\ta^+_{\om_k,-\ve}(F)- \ta^+_{0,-\ve \vert{\om_k }\vert }(F))\circ 
M_{I_{k,3}^+}(\et)}^2 \\
&\leq &
 \frac{\no\va_\iy^2}{P_k^2}\int_{\R_+} \Big(\int_{I_{k,3}^+} \va(t+ \ln( 
\vert{\om_k }\vert )- \ln(u)) M_{I_{k,3}^+} \et( \vert{\om_k }\vert 
e^t)dt\Big)^2\frac{du}{u}\\
&\leq &
\frac{\no\va_\iy^2}{P_k^2}\int_{\R_+}\big(\int_{I^+_{k, 3}}\va(t+ \ln( \vert{\om_k }\vert)- \ln(u))dt 
\int_{I_{k,3}^+} \va(t+ \ln( \vert{\om_k }\vert )- \ln(u))\vert \et( 
\vert{\om_k }\vert e^t)\vert^2 dt\big)\frac{du}{u} \\
 &\leq&
\frac{\no\va_\iy^2 \no\va_1}{P_k^2}\int_{\R_+} \int_{I_{k,3}^+} \va(t+ \ln( \vert{\om_k 
}\vert )- \ln(u))\vert \et( \vert{\om_k}\vert e^t)\vert^2 dt\frac{du}{u} \\
\nn  & \leq &
\frac{\no\va_\iy^2 \no\va_1^2}{P_k^2} \no\et^2 .
 \end{eqnarray*} 

We proceed similarly for the second relation.
\end{proof}

\begin{definition}\label{sik def om=0}
\rm   Let $(\pi_{\rh_k,\la_k})_{k\in\N} $ be a properly converging sequence with limit set $L=\Ga_1\cup\Ga_0 $ (i.e. $\limk \la_k= 0$ and $\limk \om_k= 0 $ where $\om_k:= \rh_k \la_k= \rh_k\ve \vert{\la_k }\vert $) and $\rh_k\ne0$ for $k\in\N$. We choose a sequence $(R_k)_k\subset \R_{+} $ such that
\begin{eqnarray*}
 \limk R_k= +\iy, \limk R_k^2\la_k= 0, \limk \frac{\om_k}{R^2_k\la_k}= 0,
 \end{eqnarray*}
and sequences $(Q_k)_k, (P_k)_k $ as in Definition~{\ref{Tk def}}. For $k\in\N $ and $\phi\in l^\iy(\widehat{ G})$, define the bounded linear operator $\si_k^{0,\om_k,\ve}(\phi \res L) $ (resp. $(\si_k^{ {0, \om_k,\ve}})'(\phi\res L) $) on $\l2\R$ by 
\begin{eqnarray*} 
 \si_k^{0,\om_k,\ve}(\phi\res L)&:=&
  V_k\circ  s_k^{0,\om_k, \ve, +}(\phi\res L)\circ V_k^*\oplus V_k\circ s_k^{0,\om_k,\ve,-}(\phi\res L)\circ  V_k^*,
 \end{eqnarray*}
 where 
 \begin{eqnarray*}\label{}
 \nn s_k^{0,\om_k,\ve,+}(\phi\res L) &:= & 
 \phi(\ta^+_{ \om_k,0})\circ M_{J_{k}^+}\oplus \phi(\ta^+_{0,0 })\circ M_{I_{k,2}^+}\oplus \phi(\ta^+_{0,-\ve\vert{\om_k}\vert})\circ M_{I_{k,3}^+},
 \end{eqnarray*} and
 \begin{eqnarray*}\label{}
 \nn s_k^{0,\om_k,\ve,-}(\phi\res L) &:= &
   \phi(\ta^-_{-\om_k,0})\circ M_{J_{k}^-}\oplus \phi(\ta^-_{0,0 })\circ M_{I_{k,2}^-}\oplus \phi(\ta^-_{0,\ve\vert{\om_k}\vert})\circ M_{I_{k,3}^-}.
 \end{eqnarray*}
\end{definition}

\begin{theorem}\label{lim rhklak is om0}
 Suppose that $\limk \la_k= 0$, $\limk \rh_k\la_k= 0$ and $\la_k= \ve\vert{\la_k}\vert$ for $k\in\N $. Then we have 
\begin{eqnarray*}
 & &
 \limk \noop{\pi_{\rh_k,\la_k}(a)- {\si_{k}^{0,\ve\om_k,\ve}(\hat a\res{L})}}= 0
 \end{eqnarray*}
for every $ a\in C^*(G)$.
 \end{theorem} 
 
\begin{proof} 
Apply Lemma~\ref{difference operator lak rhk ta om0} and Lemma~\ref{limk rhk lak=0}.
\end{proof}

\subsection{The final theorem}\label{fth}

Note that the representations in $\GA_1 $ do not map $C^*(G) $ into the algebra of compact operators on $\l2{\R_{+},\frac{dx}{|x|}} $. Following \cite[Definition 6.2]{Lin-Lud}, we define the {\it compact condition} on 
$l^\iy(\widehat{G}) $, which will be satisfied by the elements of $C^*(G) $.

Choose a positive valued function $q\in C_c(\R^2) $ such that $\int_{\R^2} q(x,y)dxdy= 1$. For $f\in 
\l1\R $, define the function $E(f)\in \l1{\R^3}\simeq \l1{G/\ZZ} $ by
\begin{eqnarray*}\label{}
 \nn E(f)(t, x, y):= f(t)q(x,y) \, \text{ for all } (t, x, y) \in \R^3.
 \end{eqnarray*}
We see that for any $\ell= (0, \mu,\nu, 0)$, $\mu, \nu \in \R$ and $\pi_{\mu, \nu}:= \pi_{\ell}= \pi_{(0, \mu, \nu, 0)}$:
\begin{eqnarray*}\label{}
 \nn \pi_{\mu,\nu}(E(f))(\et)(u)= \int_\R f(u-t)\hat q(\mu e^{-t},\nu 
e^t)\et(t)dt, \quad \et\in\l2\R, u\in\R.
 \end{eqnarray*}
Let $l(f)$ be the operator $(l(f)\xi)(u):= \int_{\R} f(t)\xi(u-t)dt$, for $\xi\in 
\l2\R$ and $u\in\R $, given by the regular representation of the group $(\R,+)$ on 
the Hilbert space ${\l2\R}$, and let $q_{\mu,\nu} $ be the continuous bounded 
function defined by $q_{\mu,\nu}(t):= \hat q(e^{-t}\mu, e^t \nu)$ for $t\in\R $. We thus have  
\begin{eqnarray*}\label{}
 \nn \pi_{\mu,\nu}(E(f))(\et)= l(f)(q_{\mu,\nu}\cdot \et).
 \end{eqnarray*}
 Hence, 
 \begin{eqnarray*}\label{}
 \nn \noop{\pi_{\mu,\nu}(E(f))} \leq \noop{l(f)}\no{q_{\mu,\nu}}_\iy
 \end{eqnarray*}
is bounded for $\ell \in \Ga_2\cup \Ga_1$. Clearly, $\pi_{\ell}(E(f))$ is bounded for $\ell= (\tau, 0, 0, 0) \in \Ga_0$. This shows that the mapping $E $ extends to a bounded linear mapping of $C^*(\R)$ into $C^*(G/\ZZ) $. 
 We then have the bounded linear mapping 
\begin{eqnarray}\label{Edef}
 \si_0: C_0(\R)\to C^*(G/\ZZ) \ \text{ given by } \ \si_0(\ps):= E(\F\inv(\ps)),
 \end{eqnarray}
where $\F\inv: C_0(\R)\to C^*(\R) $ is the inverse Fourier transform.

It follows from Proposition~\ref{normcontinuity} and that $C_0({\Gamma_0})= C_0(\R)$, we thus can apply the bounded linear map $\sigma_0$ to $\phi\res{\Ga_0}$ for any $\phi\in l^\iy(\wh{G})$.

\begin{definition}\label{compact condition}
\rm An operator field $\phi$ defined on $\widehat G$ is said to satisfy the \textit{compact condition} if for $\ve\in\{+1, -1\}$, the operators $\phi(\pi_{\ve,0})-\pi_{\ve,0}(\si_0(\phi\res{\Ga_0}))$ and $\phi(\pi_{0,\ve})- \pi_{0,\ve}(\si_0(\phi\res{\Ga_0}))$ are compact. Here $\pi_{\ve, 0}$ (resp. $\pi_{0, \ve}$) is the representation $\pi_{\ell}$, where $\ell= (0, \ve, 0, 0)$ (resp. $\ell= (0, 0, \ve, 0)$). 
\end{definition}

\rm 
Let $(O_{\ell_k})_k$, $\ell_k= (0, \ve\rho_k, \si, 0) \in \GA_2$ for all $ k$, be a properly converging sequence in $ \widehat {G}$, whose limit set contains the orbits $ O_{(0,\ve, 0,0)} $ and $ O_{(0, 0, \si, 0)} $. Let $(r_k)_k \subset \R$ be such that $e^{-r_k}\rho_k= 1$ for all $k \in \N$. Then $ \lim_{k\to \iy} r_k= -\iy$. Choose a positive sequence $ (\alpha_k)_k $ such that $ \alpha_k> -r_k$ for all $k\in\N$, $\lim_{k\to\iy} \alpha_k+ r_k= \iy$ and $ \lim_{k\to\iy} \frac{\alpha_k+ r_k}{r_k}= 0$. We say
that the sequence $ (\alpha_k)_k $ is  \textit{adapted} to the sequence $
(\ell_k)_k $.
\rm   For $ r \in \R $, let $ U(r) $ be the unitary operator on $\l2\R$ defined by
\begin{eqnarray}\label{}
 \nn U(r)\xi(s) &:=&\xi(s+ r) \ \mbox{for all } \xi\in\l2\R \ \mbox{and } s\in\R.
\end{eqnarray}

\begin{definition}\label{gencon}
\rm  Let $ \phi$ be an operator field defined over $\wh G$.  We say
that $ \phi $ satisfies the \textit{generic condition}(see \cite[Definition 5.10]{Lin-Lud}) if for every properly
converging sequence $ (\pi_{\ell_k})_k\subset \GA_2$ which admits limit points $ \pi_{(0,\ve,0,0)}, \pi_{(0,0,\si,0)} $ and for a sequence $( \alpha_k)_k$ adapted to the sequence $ (\ell_k)_k$, we have that
\begin{enumerate}\label{geneq}
\item \begin{eqnarray}
\nn \lim_{k\to\iy}\noop{U(r_k)\circ \phi(\pi_{\ell_k})\circ U(-r_k)\circ
M_{(-\iy,\alpha_k)}- \phi(\pi_{0,\ve,0,0})\circ M_{(-\iy,\alpha_k)}}&=&0,
\end{eqnarray}
\item \begin{eqnarray}\label{}
\nn \lim_{k\to\iy}\noop{U(r_k)\circ \phi(\pi_{\ell_k})\circ U(-r_k)\circ
M_{(-\alpha_k,\iy)}-\phi(\pi_{0,0,\si, 0)}\circ M_{(-\alpha_k,\iy)}}&=&0.
\end{eqnarray}
\end{enumerate}
\end{definition}

\begin{remark}\label{gensat}
\rm   According to Proposition 5.12 in \cite{Lin-Lud}, every $\phi \in \widehat{C^*(G/\ZZ)} $ satisfies the generic condition as $G/\ZZ$ is an $ax+b$-like group.
 \end{remark}

We identify as before the spectrum $\widehat{G}$ of $G$ with the set 
\begin{align*}
\GA= \Sigma_3 \cup \Sigma_2\cup \Sigma_1\cup \Sigma_0,
\end{align*}
and the C*-algebra $l^\iy(\widehat{G}) $ is then the algebra of all uniformly bounded operator fields $\ph$ defined on $\widehat{G}$ with values in $B(\l2\R) $ on $\Sigma_3 $, values in $\B(\l2{\R_+, \frac{dx}{|x|}})$ on $\Sigma_2\cup\Sigma_1 $, and with values in $\C $ on $\Sigma_0 $. We define in the following the subset $D^*(G)$ of $ l^\iy(\wh{G}) $ which will be our desired C*-algebra of Boidol's group.

\begin{definition}\label{def:DstarG}
\rm Let $ D^*(G) $ be the subset of $ l^\iy(\wh{G}) $ consisting of all the 
operator fields $ \ph $ defined over $ \wh{G} $ with the following properties:
\begin{enumerate}\label{}
\item  the mapping $\ga\mapsto \ph(\ga) $ vanishes at infinity,


\item
 \begin{enumerate}
 \item the mapping $\ga\mapsto \ph(\pi_\ga)$ is norm continuous on the set $\Sigma_3$,
 \item for any $\ga \in \Sigma_3 $ the operator $\ph(\ga) $ is compact,
 \item for every properly converging sequence 
$(\pi_{\rh_k,\la_k})_k$ in $\Sig_3 $ with limit set $L\in \Sig_2 $ (i.e. when $\limk 
\om_k= \om\ne 0 $, where $\om_k= \rh_k \la_k$ and $\la_k= \ve |\la_k|$), 
for the sequence $(R_k)_k\subset \R_{+} $ with the 
properties given in (\ref{Rk conditions 2- nonzero om}) and the operator 
$\si^\om_k $ defined in (\ref{control om}), 
we have that
\begin{eqnarray*}\label{}
 \nn \limk\noop{\ph(\pi_{\rh_k, \la_k})- \si^\om_k(\ph\res{L})}=  0,
 \end{eqnarray*}
\item for every properly converging sequence $(\pi_{\rh_k,\la_k})_k$ in $\Sig_3 $ with limit set $L=  \Sig_1\cup\Sig_0 $, (i.e. when $\limk \om_k= 0 $), admitting a sub-sequence, we have
\begin{eqnarray*}\label{}
 \nn \limk\noop{\ph(\pi_{\rh_k,\la_k})- \si_k^{0,\ve\om_k,\ve}(\ph\res{L})}= 0,
 \end{eqnarray*}
 where the sequences $(R_k), (P_k) $ and $(Q_k) $ are as in Definition~\ref{Tk def},
 \end{enumerate}

\item  
 \begin{enumerate}\label{}
\item the mappings $\ga\mapsto \ph(\pi_\ga)$ are norm continuous on the sets $\Sig_j$ for $j= 0, 1, 2$ (in particular the function 
$\ph\res{\Sig_0} $ is in $C_0(\Sig_0)= C_0(\R) \simeq C^*(\R)$),
 \item for any $\ga\in \Sig_2$, the operator $\ph(\ga) $ is compact,
 \item for every properly converging sequence $(\ta_{\om_k,-\ve})_k$ in $\Sig_2$, 
$\ve\in\{+1, -1\}$, with limit set $L= \Sig_1\cup \Sig_0 $ (i.e. $\limk \om_k= 0$), 
for the sequence $(R_k)\subset \R_{+} $ with the 
properties in Definition \ref{sik def om=0}, we have  that  
{\begin{eqnarray*}\label{}
& &
\limk \noop{ \ph(\ta^+_{\om_k,-\ve })- {s_k^{0,\om_k,\ve,+}(\ph\res{L})}}= 0\\
  \text{and }\\
& &
\limk \noop{ \ph(\ta^-_{-\om_k,\ve })- {s_k^{0,\om_k,\ve,-}(\ph\res{L})}}= 0,
 \end{eqnarray*}}
  \item  the operator $\phi(\gamma)$, $\ga\in \Sig_2$, satisfies the generic condition given in 
Definition \ref{gencon},
 \item  the operator $\ph(\ga)$, $\ga\in \Sig_1$, satisfies the compact 
condition given in Definition \ref{compact condition}, 
\end{enumerate}

\item  the adjoint operator field $\ph^* $ satisfies the same conditions.
 \end{enumerate}
 \end{definition}

\begin{remark}
The operators $\sigma_k^{0, \ve\om_k, \ve}$, $s_k^{0, \om_k, \ve, \pm}$ and 
$\si_k^{\om}$ are defined in Definition~\ref{sik def om=0} and \eqref{control 
om}, respectively. Note that due to the limit set of a properly converging 
sequence may lie in different $\Sigma_j$, $j= 0, 1, 2$, the operators involved 
in the approximations will depend on sequences $(R_k)_k$ in $\R_{+}$ with 
different conditions given previously, and on sequences $(Q_k)_k, (P_k)_k$ which 
give rise the internals $I^{\pm}_{k, 2}, J^{\pm}_{k}$ used in 
Definition~\ref{sik def om=0}. 
\end{remark}

\rm
We will show that the set $D^*(G)$ defined in Definition~\ref{def:DstarG} is a C*-subalgebra of $l^\iy(\widehat{G})$. 


\begin{remark}\label{cstar GZZ}
\rm   Since $\ZZ$ is the centre of Boidol's group $G$, we have seen in Remark \ref{closed} that the spectrum of $G/\ZZ$ (thus, of $C^*(G/\ZZ)$) can be identified with the subset $S_2:= \GA_2\cup \GA_1\cup \GA_0 $ of the coadjoint orbit space, and the group $G/\ZZ$ is $ax+ b$-like. It has 
been shown in \cite[Section 8]{Lin-Lud} that the family of uniformly bounded operator 
fields defined over the set $S_2:= \Sig_2\cup \Sig_1\cup \Sig_0 $ satisfying 
the conditions $(1) $, $(3)$ and $(4)$ in Definition~\ref{def:DstarG} forms a C*-algebra (\cite[Proposition 8.2]{Lin-Lud}), denoted by $D^*(G/\ZZ)$ for our purpose, and it is isomorphic to $C^*(G/\ZZ) $ via the 
Fourier transform. Hence the natural restriction map
$R_{G/\ZZ}$: $\widehat{C^*(G)}\to \widehat{C^*(G/\ZZ)}$
is surjective, so is the natural quotient map $P_{G/\ZZ}:C^*(G)\to C^*(G/\ZZ) 
$ (see Remark \ref{closed}(2)). This shows that for every operator field $\phi\in D^*(G/\ZZ)$, 
there exists an $a_\phi\in C^*(G) $ such that $\phi$ coincides with $\widehat{a_\phi}\res{S_2}$. That is, for every $\phi\in D^*(G/\ZZ)= \F_{G/\ZZ}(C^*(G/\ZZ))$, where $\F_{G/\ZZ}$ is the Fourier transform defined on the C*-algebra of $G/\ZZ$ (see also \cite[Definition 5.3]{Lin-Lud}), thanks to the isomorphism $C^*(G/\ZZ)\simeq C^*(G)/K_{S_2}$ (given in Remark \ref{closed}(2)) and the spectrum $\widehat{C^*(G)}$ can be identified with $\bigcup_{i= 0}^3 \Sigma_i$ which contains $S_2$, we have such an $a_{\phi}\in C^*(G)$.
 \end{remark}

Note that it follows from the preceding sections, mainly Proposition~\ref{normcontinuity}, Theorems~\ref{lim rhklak is om non0}, 
\ref{lim rhklak is om0} and Remarks~\ref{gensat} and \ref{cstar GZZ}, that $\F(a)= \hat{a}$, $a\in C^*(G)$, satisfies all the conditions in $D^*(G)$. Thus, $\widehat{C^*(G)}\subset D^*(G)$.

\begin{lemma}\label{Dstar complete}
The set $D^*(G)$ is a complete subspace of $l^\iy(\widehat{G})$.
\end{lemma}

\begin{proof}
Let $(\phi_n)_n \subset D^*(G)$ be a Cauchy sequence. Then for every $\ga \in 
\widehat{G}$, the sequence $(\phi_n(\ga))_n$ is a Cauchy sequence  
in the C*-algebra of bounded linear operators on the corresponding Hilbert 
spaces $\H_\ga $, and so $\ph(\ga)= \limk \ph_k(\ga) $ exists in $\B(\H_\ga) $. 
By Remark~\ref{cstar GZZ}, there is an operator field $\phi$ 
such that $\phi|_{S_2}$ is contained in $C^*(G/\ZZ)$ and 
$\noop{\phi_n(\ga)- \phi(\ga)} \to 0$ as $n\to \iy$ for every $\ga\in S_2$. 
It then suffices to show that $\phi$ satisfies the conditions (2) in Definition~\ref{def:DstarG}. It is clear that $\phi$ satisfies (2-a) and (2-b). For any properly converging sequence $(\pi_{\rho_k, \la_k})_k$ in $\Sig_3$ with limit set in $\Sig_2$, we have that $\lim_{k \to \iy} \noop{\phi_n(\pi_{\rho_k, \la_k})- \si_k^{\om}(\phi_n)}= 0$ for all $n$ (since $\phi_n$ is in $D^*(G)$). Hence, 
\begin{eqnarray*}
\noop{\phi(\pi_{\rho_k, \la_k})- \si_k^{\om}(\phi \res{L})} & \leq & \noop{\phi(\pi_{\rho_k, \la_k})- \phi_n(\pi_{\rho_k, \la_k})}+ \noop{\phi_n(\pi_{\rho_k, \la_k})- \si_k^{\om}({\phi_n}\res{L})} \\
 & & + \noop{\si_k^{\om}({\phi_n}\res{L})- \si_k^{\om}(\phi\res{L})}
\end{eqnarray*}
which converges to zero, thus, (2-c) is satisfied by the operator field $\phi$. Similar arguments hold for the condition (2-d). Therefore, we have that $\phi \in D^*(G)$. 
\end{proof}

\begin{lemma}\label{Dstar is cstar}
The subspace $D^*(G) $ of $l^\iy(\widehat{G}) $ is a postliminal C*-algebra with spectrum equal to $\widehat{G} $.
\end{lemma}
 
\begin{proof} 
It is shown in Lemma~\ref{Dstar complete} that $D^*(G) $ is a closed subspace 
of $l^\iy(\widehat{G})$ and it is clear that $D^*(G)$ is involutive. Let us show 
that it is also a subalgebra. Let $\phi, \phi' $ be two elements 
of $D^*(G) $. By Remark~\ref{cstar GZZ}, there exist $a, a'\in C^*(G/\ZZ) $ such 
that $\phi\res{S_2}= \hat a$ and $\phi'\res{S_2}= \hat a'$. 
Therefore we only need to check if the conditions (2-b), (2-c) and (2-d) work for the  
product $\phi \circ \phi'$. Condition (2-b) is evident.
Now for any $\om\in{ \R^*} $ by Theorem~\ref{lim rhklak is om non0}, we have that
\begin{eqnarray*}\label{}
 \nn \limk\noop{\si_k^\om((\widehat{a\cdot a'})\res{L} )- \si_k^\om(\hat a\res{L})\circ \si_k^\om( \hat a'\res{L})}= 0,
 \end{eqnarray*} 
 and so 
\begin{eqnarray*}\label{}
 \nn \limk \noop{\phi\circ \phi'(\pi_{\rh_k,\la_k})- {\si_k^\om({\phi\circ \phi'\res{L}})}} 
&= &
 \limk\noop{\phi\circ \phi'(\pi_{\rh_k,\la_k})- \si_k^\om((\widehat {a\cdot a'}) \res{L})}\\
\nn  &= &
 \limk\noop{\phi(\pi_{\rh_k,\la_k})\circ \phi'(\pi_{\rh_k,\la_k})-\si_k^\om(\hat a \res{L})\circ \si_k^\om(\hat a' \res{L}) }\\
\nn  &= &
0.
 \end{eqnarray*}
Condition (2-d) follows in  a similar way thanks to Theorem \ref{lim rhklak is om0}.

Hence $D^*(G)$, being a closed $*$-subalgebra of $l^{\iy}(\hat G)$, is  a 
C*-algebra which contains the Fourier transform of  $C^*(G) $. It is clear that 
the spectrum of $D^*(G) $ contains $\widehat{G} $ since point evaluations are 
irreducible for $\widehat{C^*(G)} $. Take now $\pi\in\widehat{D^*(G)} $. Since there is a natural restriction map from $\widehat{C^*(G)}$ onto $\widehat{C^*(G/\ZZ)}$ by Remark \ref{cstar GZZ}, and the Fourier transfer defined on $C^*(G)$ (resp. on $C^*(G/\ZZ)$) is an isometric homomorphism onto $\mathcal{F}_G(C^*(G))$ which is contained in $D^*(G)$ (resp. $\mathcal{F}_{G/\ZZ}(C^*(G/\mathcal{Z}))$ which coincides with $D^*(G/\ZZ)$), thus, there is a restriction map, denoted by $R^0$, from $D^*(G)$ to $D^*(G/\ZZ) $ which is also surjective. Let $K^0= \{ \ph\in D^*(G); \ph(\ga)=0 \, \text{ for } 
\ga\in S_2\}$. It can be seen easily that $K^0$ is an ideal of $D^*(G)$. Note that $D^*(G/\ZZ)\simeq C^*(G/\ZZ)\simeq C^*(G)/K_{S_2}$.
It follows immediately from the conditions (1) and (2) that the 
ideal $K^0 $ of $D^*(G) $ is just the algebra $C_0(\GA_3,\K) $ of continuous 
mappings defined on the locally compact space $\GA_3 $ with values in the 
algebra $\K $ of compact operators on $\l2\R$. If $\pi(K^0)= (0) $, then  $\pi $ 
can be identified with an irreducible representation of $C^*(G/\ZZ) $, so there 
exists $\pi' \in S_2$ such that $\pi(\phi)= \phi(\pi')$ for $\phi\in D^*(G) $. 
If $\pi(K^0)\ne (0) $, then $\pi $ defines an irreducible  representation of the 
algebra $C_0(\GA_3,\K) $ and hence there exists $\pi_{\rh,\la}\in \Sig_3 $ such 
that $\pi(\phi)= \phi(\pi_{\rh,\la})$ for $\phi\in K^0 $, and finally $\pi $ is 
the evaluation at $\pi_{\rh,\la} $ for every $\phi\in D^*(G) $.

It is clear that $D^*(G) $ is postliminal, since every (non-trivial)
irreducible representation $(\pi,\H_{\pi}) $ of $D^*(G) $ induces compact 
operators on $\H_{\pi}$ (conditions (2-b), (3-b), (3-e)).
\end{proof}

Now we have our main theorem which characterises the C*-algebra $C^*(G)$ of Boidol's group by the above descriptions of the Fourier transform of $C^*(G)$ onto $D^*(G)$.

\begin{theorem}\label{Cster G identified}
The Fourier transform defined in (\ref{fourrier def}) is an isomorphism of the C*-algebra of Boidol's group $G $ onto the C*-algebra $D^*(G) $. 
\end{theorem}

\begin{proof} Using Lemma~\ref{Dstar is cstar}, the Stone-Weierstrass theorem for C*-algebras (see \cite[11.1.8]{Di}) tells us  that the subalgebra $\widehat{C^*(G)} $ of $D^*(G) $ is equal to $D^*(G) $. 
\end{proof}


\begin{thebibliography}{1000}
\bibitem{Be-Be-Lu} I. Beltita, D. Beltita, J. Ludwig, 
Fourier Transforms of C*-Algebras of Nilpotent Lie Groups, 
{\em Int. Math. Res. Not. IMRN} {3} (2017), 677-714.

\bibitem{Boi} J. Boidol, $ *$-regularity of exponential Lie groups. {\em Invent. 
Math}. {56} (1980), no. 3, 231-238. 

\bibitem{Boi2} J. Boidol, On a regularity condition for group algebras of non-abelian locally compact groups. Harmonic analysis, Iraklion 1978 (Proc. Conf., Univ. Crete, Iraklion, 1978), pp. 16-21, Lecture Notes in Math., 781, Springer, Berlin, 1980. 

\bibitem{Boi1} J. Boidol, $* $-regularity of some classes of solvable groups. 
{\em Math. Ann}. {261} (1982), no. 4, 477-481. 

\bibitem{Fe} J. M. G. Fell, The structure of algebras of operator fields, {\em 
Acta Math.}  {106}  (1961),  233-280.



\bibitem {ILL1} J. Inoue, Y.-F. Lin, J. Ludwig, The solvable Lie group 
$N_{6,28}$: an example of an almost $C_0(\K)$-C*-algebra. {\em Adv. Math.} {272} (2015), 
252-307. 
   
\bibitem {ILL2} J. Inoue, Y.-F. Lin, J. Ludwig, A class of almost $C_0(\K)$-C*-algebras. {\em J. Math. Soc. Japan} {68} (2016), no. 1, 
71-89. 

\bibitem {Di} J. Dixmier, C*-algebras. Translated from the French
by Francis Jellett. North-Holland Mathematical Library, Vol. 15.
North-Holland Publishing Co., Amsterdam-New York-Oxford, 1977.
xiii+492 pp.

\bibitem{Lee1} R.-Y. Lee, On the C* Algebras of Operator Fields {\em Indiana
Univ. Math. J.} {26} (1977) no. 2, 351-372.

\bibitem {Lin-Lud} Y.-F. Lin, J. Ludwig, The C*-algebra of $ax+b$-like groups,
\emph{J. Funct. Anal.} {259} (2010), 104-130.

\bibitem{Lep-Lud} H. Leptin, J. Ludwig, Unitary representation theory  
of exponential Lie groups, {\em De Gruyter Expositions in Mathematics} 18, 1994.

\bibitem {Lu-Tu} J. Ludwig, L. Turowska, The C*-algebras of the
Heisenberg Group and of thread-like Lie groups, {\em Math. Z.} {268} 
(2011), no. 3-4, 897-930.

\bibitem{Lu-Re}  J. Ludwig, H. Regeiba, The C*-algebra of 
the Heisenberg motion groups $\T^n\ltimes \H_n$, 
{\em Complex Anal. Operator Theory}, 13 (2019), no. 8, 3943-3978.


\bibitem{Reiter} H. Reiter, J. D. Stegeman, Classical Harmonic analysis 
and locally compact groups, {\em Clarendon Press}, 2000.

\end{thebibliography}
\end{document}